\documentclass[11pt]{article}
\usepackage{latexsym,amssymb,amsmath,amsthm,enumerate,geometry,cite,float,verbatim}
\newtheorem{theorem}{Theorem}
\newtheorem{lemma}[theorem]{Lemma}
\newtheorem{conjecture}[theorem]{Conjecture}
\newtheorem{proposition}[theorem]{Proposition}

\newtheorem{claim}{Claim}

\usepackage{tikz}
\usepackage{lineno}
\usepackage{setspace}
\allowdisplaybreaks

\begin{document}
\onehalfspace

\title{Relating the independence number and the dissociation number}
\author{Felix Bock\and Johannes Pardey\and Lucia D. Penso\and Dieter Rautenbach}
\date{}

\maketitle
\vspace{-10mm}
\begin{center}
{\small 
Institute of Optimization and Operations Research, Ulm University, Ulm, Germany\\
\texttt{$\{$felix.bock,johannes.pardey,lucia.penso,dieter.rautenbach$\}$@uni-ulm.de}
}
\end{center}

\begin{abstract}
The independence number $\alpha(G)$ and 
the dissociation number ${\rm diss}(G)$
of a graph $G$ are the largest orders of induced subgraphs of $G$ 
of maximum degree at most $0$ and at most $1$, respectively. 
We consider possible improvements of the obvious inequality
$2\alpha(G)\geq {\rm diss}(G)$.
For connected cubic graphs $G$ distinct from $K_4$, 
we show $5\alpha(G)\geq 3{\rm diss}(G)$, and 
describe the rich and interesting structure of the extremal graphs in detail.
For bipartite graphs, and, more generally, triangle-free graphs,
we also obtain improvements.
For subcubic graphs though, the inequality cannot be improved in general,
and we characterize all extremal subcubic graphs.\\[3mm]
{\bf Keywords:} Dissociation set, independent set, cubic graph, triangle-free gaph
\end{abstract}

\section{Introduction}

We consider finite, simple, and undirected graphs, and use standard terminology.
A set $D$ of vertices of a graph $G$
is a {\it dissociation set} in $G$ if the subgraph $G[D]$ of $G$ induced by $D$
has maximum degree at most $1$, and
the {\it dissociation number ${\rm diss}(G)$} of $G$ 
is the maximum order of a dissociation set in $G$.
Clearly, dissociation sets and the dissociation number bear a close resemblance to
the well known independent sets and the independence number $\alpha(G)$ 
of a graph $G$.
Indeed, if $D$ is a dissociation set in a graph $G$, then every component of $G[D]$ has one or two vertices.
Hence, selecting one vertex in each component yields an independent set 
of order at least $|D|/2$, which implies
\begin{eqnarray}\label{e0-1}
\alpha(G)\geq \frac{1}{2}{\rm diss}(G)\mbox{ for every graph $G$.}
\end{eqnarray}
While (\ref{e0-1}) is trivial, 
it is NP-hard \cite{bopapera1} to recognize the extremal graphs for (\ref{e0-1}).
The dissociation number is algorithmically hard even when restricted, 
for instance, to subcubic bipartite graphs \cite{bocalo,ordofigowe,ya}.
Bounds \cite{bopapera2,brkakase,brjakaseta,goharasc},
fast exact algorithms \cite{kakasc}, 
(randomized) approximation algorithms \cite{kakasc,hobu}, 
fixed parameter tractability \cite{ts},
and the maximum number of maximum dissociation sets \cite{tulidu}
have been studied for this parameter or its dual, 
the {\it $3$-path (vertex) cover} number.

The starting point for the present paper were possible improvements of (\ref{e0-1}).
As it turns out, (sub)cubic graphs play a special role in this context,
and it is not difficult to construct arbitrarily large connected $r$-regular extremal graphs for (\ref{e0-1}) as soon as $r\geq 4$, cf.~Section \ref{section2}.
For cubic graphs though, 
our first main result is the following best possible improvement of (\ref{e0-1}).
Note that the connected cubic graph $K_4$ satisfies (\ref{e0-1}) with equality.

\begin{theorem}\label{theorem1}
If $G$ is a connected cubic graph of order at least $6$, then $\alpha(G)\geq \frac{3}{5}{\rm diss}(G)$.
\end{theorem}
The extremal graphs for Theorem \ref{theorem1} have a very interesting structure, 
which we elucidate in Theorem \ref{conjecture1} 
at the end of Section \ref{section2}.
Their order is necessarily divisible by $18$, 
and Figure \ref{fig3} shows a smallest extremal graph.

\begin{figure}[H]
\begin{center}
\unitlength .4mm 
\linethickness{0.4pt}
\ifx\plotpoint\undefined\newsavebox{\plotpoint}\fi 
\begin{picture}(132,102)(0,0)
\put(0,80){\circle*{2}}
\put(130,80){\circle*{2}}
\put(0,20){\circle*{2}}
\put(130,20){\circle*{2}}
\put(20,80){\circle*{2}}
\put(110,80){\circle*{2}}
\put(30,40){\circle*{2}}
\put(100,40){\circle*{2}}
\put(20,20){\circle*{2}}
\put(110,20){\circle*{2}}
\put(20,100){\circle*{2}}
\put(110,100){\circle*{2}}
\put(30,60){\circle*{2}}
\put(100,60){\circle*{2}}
\put(20,0){\circle*{2}}
\put(110,0){\circle*{2}}
\put(20,100){\line(0,-1){20}}
\put(110,100){\line(0,-1){20}}
\put(30,60){\line(0,-1){20}}
\put(100,60){\line(0,-1){20}}
\put(20,0){\line(0,1){20}}
\put(110,0){\line(0,1){20}}
\put(20,80){\line(-1,0){20}}
\put(110,80){\line(1,0){20}}
\put(20,20){\line(-1,0){20}}
\put(110,20){\line(1,0){20}}
\put(0,80){\line(1,1){20}}
\put(130,80){\line(-1,1){20}}
\put(0,20){\line(1,-1){20}}
\put(130,20){\line(-1,-1){20}}
\put(55,50){\circle*{2}}
\put(75,50){\circle*{2}}
\put(55,50){\line(-5,2){25}}
\put(75,50){\line(5,2){25}}
\put(55,50){\line(-5,-2){25}}
\put(75,50){\line(5,-2){25}}
\put(0,81){\line(0,-1){61}}
\put(130,81){\line(0,-1){61}}
\put(75,50){\line(-1,0){20}}
\put(110,0){\line(-1,0){90}}
\put(55,50){\circle{4}}
\put(75,50){\circle{4}}
\put(20,80){\circle{4}}
\put(20,100){\circle{4}}
\put(130,80){\circle{4}}
\put(110,80){\circle{4}}
\put(110,20){\circle{4}}
\put(110,0){\circle{4}}
\put(20,20){\circle{4}}
\put(0,20){\circle{4}}
\put(65,100){\vector(1,0){.176}}\put(20,100){\line(1,0){90}}
\put(130,50){\vector(0,-1){.176}}\put(130,80){\line(0,-1){60}}
\put(65,0){\vector(-1,0){.176}}\put(110,0){\line(-1,0){90}}
\put(0,50){\vector(0,1){.176}}\put(0,20){\line(0,1){60}}
\put(25,70){\vector(1,-2){.176}}\multiput(20,80)(.084033613,-.168067227){119}{\line(0,-1){.168067227}}
\put(105,70){\vector(-1,-2){.176}}\multiput(110,80)(-.084033613,-.168067227){119}{\line(0,-1){.168067227}}
\put(25,30){\vector(1,2){.176}}\multiput(20,20)(.084033613,.168067227){119}{\line(0,1){.168067227}}
\put(105,30){\vector(-1,2){.176}}\multiput(110,20)(-.084033613,.168067227){119}{\line(0,1){.168067227}}
\end{picture}
\end{center}
\caption{A cubic graph $G$ with 
$\alpha(G)=6$ and ${\rm diss}(G)=10$;
the encircled vertices indicate a maximum dissociation set.
The orientation of some of the edges of $G$ 
is explained in the context of Theorem \ref{conjecture1} below.}\label{fig3}
\end{figure}
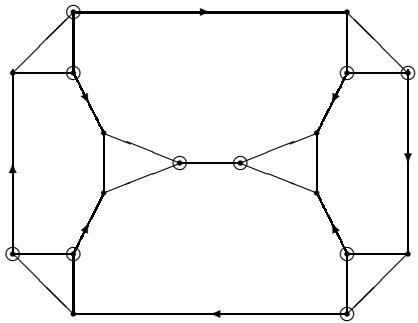

For subcubic graphs, the inequality (\ref{e0-1}) cannot be improved. 
Nevertheless, the subcubic extremal graphs have a simple structure,
and, as our second main result, we provide their constructive characterization:
Let the set ${\cal G}$ of connected graphs contain $K_4$ as well as all connected subcubic graphs that arise 
from the union of disjoint copies of the following three graphs, in which we mark certain vertices:
\begin{itemize}
\item $K_2$ with both vertices marked,
\item $K_3$ with two of the three vertices marked, and
\item the graph $K_4^*$ that arises from $K_4$ with vertices $a$, $b$, $c$, and $d$ 
by subdiviging the edge $ab$ twice, and
marking the two vertices created by these two subdivisions as well as the vertices $c$ and $d$,
\end{itemize}
by adding additional edges, each incident with at most one marked vertex.
For a graph $G$ in ${\cal G}\setminus \{ K_4\}$, let $D(G)$ be the set of all marked vertices.
See Figure \ref{fig2} for an illustration.

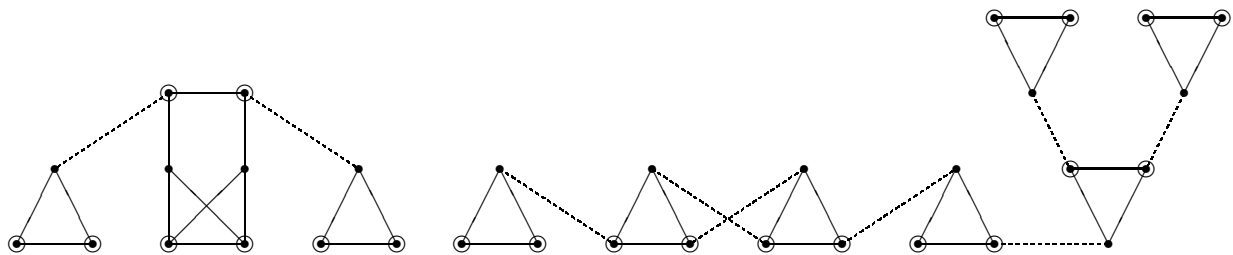
\begin{figure}[H]
\begin{center}
\unitlength 0.5mm 
\linethickness{0.4pt}
\ifx\plotpoint\undefined\newsavebox{\plotpoint}\fi 
\begin{picture}(112,52)(0,0)
\put(50,10){\circle*{2}}
\put(90,10){\circle*{2}}
\put(10,10){\circle*{2}}
\put(50,30){\circle*{2}}
\put(50,50){\circle*{2}}
\put(70,10){\circle*{2}}
\put(110,10){\circle*{2}}
\put(30,10){\circle*{2}}
\put(70,30){\circle*{2}}
\put(70,50){\circle*{2}}
\put(50,50){\line(1,0){20}}
\put(70,50){\line(0,-1){40}}
\put(70,10){\line(-1,1){20}}
\put(50,30){\line(0,-1){20}}
\put(50,10){\line(1,1){20}}
\put(70,10){\line(-1,0){20}}
\put(110,10){\line(-1,0){20}}
\put(30,10){\line(-1,0){20}}
\put(50,50){\line(0,-1){20}}
\put(70,50){\circle{4}}
\put(50,50){\circle{4}}
\put(50,10){\circle{4}}
\put(90,10){\circle{4}}
\put(10,10){\circle{4}}
\put(70,10){\circle{4}}
\put(110,10){\circle{4}}
\put(30,10){\circle{4}}
\put(100,30){\circle*{2}}
\put(20,30){\circle*{2}}
\put(110,10){\line(-1,2){10}}
\put(30,10){\line(-1,2){10}}
\put(100,30){\line(-1,-2){10}}
\put(20,30){\line(-1,-2){10}}
\multiput(69.93,49.93)(.0493421,-.0328947){16}{\line(1,0){.0493421}}
\multiput(71.509,48.877)(.0493421,-.0328947){16}{\line(1,0){.0493421}}
\multiput(73.088,47.824)(.0493421,-.0328947){16}{\line(1,0){.0493421}}
\multiput(74.667,46.772)(.0493421,-.0328947){16}{\line(1,0){.0493421}}
\multiput(76.245,45.719)(.0493421,-.0328947){16}{\line(1,0){.0493421}}
\multiput(77.824,44.667)(.0493421,-.0328947){16}{\line(1,0){.0493421}}
\multiput(79.403,43.614)(.0493421,-.0328947){16}{\line(1,0){.0493421}}
\multiput(80.982,42.561)(.0493421,-.0328947){16}{\line(1,0){.0493421}}
\multiput(82.561,41.509)(.0493421,-.0328947){16}{\line(1,0){.0493421}}
\multiput(84.14,40.456)(.0493421,-.0328947){16}{\line(1,0){.0493421}}
\multiput(85.719,39.403)(.0493421,-.0328947){16}{\line(1,0){.0493421}}
\multiput(87.298,38.351)(.0493421,-.0328947){16}{\line(1,0){.0493421}}
\multiput(88.877,37.298)(.0493421,-.0328947){16}{\line(1,0){.0493421}}
\multiput(90.456,36.245)(.0493421,-.0328947){16}{\line(1,0){.0493421}}
\multiput(92.035,35.193)(.0493421,-.0328947){16}{\line(1,0){.0493421}}
\multiput(93.614,34.14)(.0493421,-.0328947){16}{\line(1,0){.0493421}}
\multiput(95.193,33.088)(.0493421,-.0328947){16}{\line(1,0){.0493421}}
\multiput(96.772,32.035)(.0493421,-.0328947){16}{\line(1,0){.0493421}}
\multiput(98.351,30.982)(.0493421,-.0328947){16}{\line(1,0){.0493421}}
\multiput(49.93,49.93)(-.0493421,-.0328947){16}{\line(-1,0){.0493421}}
\multiput(48.351,48.877)(-.0493421,-.0328947){16}{\line(-1,0){.0493421}}
\multiput(46.772,47.824)(-.0493421,-.0328947){16}{\line(-1,0){.0493421}}
\multiput(45.193,46.772)(-.0493421,-.0328947){16}{\line(-1,0){.0493421}}
\multiput(43.614,45.719)(-.0493421,-.0328947){16}{\line(-1,0){.0493421}}
\multiput(42.035,44.667)(-.0493421,-.0328947){16}{\line(-1,0){.0493421}}
\multiput(40.456,43.614)(-.0493421,-.0328947){16}{\line(-1,0){.0493421}}
\multiput(38.877,42.561)(-.0493421,-.0328947){16}{\line(-1,0){.0493421}}
\multiput(37.298,41.509)(-.0493421,-.0328947){16}{\line(-1,0){.0493421}}
\multiput(35.719,40.456)(-.0493421,-.0328947){16}{\line(-1,0){.0493421}}
\multiput(34.14,39.403)(-.0493421,-.0328947){16}{\line(-1,0){.0493421}}
\multiput(32.561,38.351)(-.0493421,-.0328947){16}{\line(-1,0){.0493421}}
\multiput(30.982,37.298)(-.0493421,-.0328947){16}{\line(-1,0){.0493421}}
\multiput(29.403,36.245)(-.0493421,-.0328947){16}{\line(-1,0){.0493421}}
\multiput(27.824,35.193)(-.0493421,-.0328947){16}{\line(-1,0){.0493421}}
\multiput(26.245,34.14)(-.0493421,-.0328947){16}{\line(-1,0){.0493421}}
\multiput(24.667,33.088)(-.0493421,-.0328947){16}{\line(-1,0){.0493421}}
\multiput(23.088,32.035)(-.0493421,-.0328947){16}{\line(-1,0){.0493421}}
\multiput(21.509,30.982)(-.0493421,-.0328947){16}{\line(-1,0){.0493421}}
\end{picture}
\linethickness{0.4pt}
\ifx\plotpoint\undefined\newsavebox{\plotpoint}\fi 
\begin{picture}(212,72)(0,0)
\put(90,10){\circle*{2}}
\put(170,30){\circle*{2}}
\put(150,70){\circle*{2}}
\put(190,70){\circle*{2}}
\put(50,10){\circle*{2}}
\put(130,10){\circle*{2}}
\put(10,10){\circle*{2}}
\put(110,10){\circle*{2}}
\put(190,30){\circle*{2}}
\put(170,70){\circle*{2}}
\put(210,70){\circle*{2}}
\put(70,10){\circle*{2}}
\put(150,10){\circle*{2}}
\put(30,10){\circle*{2}}
\put(110,10){\line(-1,0){20}}
\put(190,30){\line(-1,0){20}}
\put(170,70){\line(-1,0){20}}
\put(210,70){\line(-1,0){20}}
\put(70,10){\line(-1,0){20}}
\put(150,10){\line(-1,0){20}}
\put(30,10){\line(-1,0){20}}
\put(90,10){\circle{4}}
\put(170,30){\circle{4}}
\put(150,70){\circle{4}}
\put(190,70){\circle{4}}
\put(50,10){\circle{4}}
\put(130,10){\circle{4}}
\put(10,10){\circle{4}}
\put(110,10){\circle{4}}
\put(190,30){\circle{4}}
\put(170,70){\circle{4}}
\put(210,70){\circle{4}}
\put(70,10){\circle{4}}
\put(150,10){\circle{4}}
\put(30,10){\circle{4}}
\put(100,30){\circle*{2}}
\put(180,10){\circle*{2}}
\put(160,50){\circle*{2}}
\put(200,50){\circle*{2}}
\put(60,30){\circle*{2}}
\put(140,30){\circle*{2}}
\put(20,30){\circle*{2}}
\put(110,10){\line(-1,2){10}}
\put(190,30){\line(-1,-2){10}}
\put(170,70){\line(-1,-2){10}}
\put(210,70){\line(-1,-2){10}}
\put(70,10){\line(-1,2){10}}
\put(150,10){\line(-1,2){10}}
\put(30,10){\line(-1,2){10}}
\put(100,30){\line(-1,-2){10}}
\put(180,10){\line(-1,2){10}}
\put(160,50){\line(-1,2){10}}
\put(200,50){\line(-1,2){10}}
\put(60,30){\line(-1,-2){10}}
\put(140,30){\line(-1,-2){10}}
\put(20,30){\line(-1,-2){10}}
\multiput(19.93,29.93)(.0493421,-.0328947){16}{\line(1,0){.0493421}}
\multiput(21.509,28.877)(.0493421,-.0328947){16}{\line(1,0){.0493421}}
\multiput(23.088,27.824)(.0493421,-.0328947){16}{\line(1,0){.0493421}}
\multiput(24.667,26.772)(.0493421,-.0328947){16}{\line(1,0){.0493421}}
\multiput(26.245,25.719)(.0493421,-.0328947){16}{\line(1,0){.0493421}}
\multiput(27.824,24.667)(.0493421,-.0328947){16}{\line(1,0){.0493421}}
\multiput(29.403,23.614)(.0493421,-.0328947){16}{\line(1,0){.0493421}}
\multiput(30.982,22.561)(.0493421,-.0328947){16}{\line(1,0){.0493421}}
\multiput(32.561,21.509)(.0493421,-.0328947){16}{\line(1,0){.0493421}}
\multiput(34.14,20.456)(.0493421,-.0328947){16}{\line(1,0){.0493421}}
\multiput(35.719,19.403)(.0493421,-.0328947){16}{\line(1,0){.0493421}}
\multiput(37.298,18.351)(.0493421,-.0328947){16}{\line(1,0){.0493421}}
\multiput(38.877,17.298)(.0493421,-.0328947){16}{\line(1,0){.0493421}}
\multiput(40.456,16.245)(.0493421,-.0328947){16}{\line(1,0){.0493421}}
\multiput(42.035,15.193)(.0493421,-.0328947){16}{\line(1,0){.0493421}}
\multiput(43.614,14.14)(.0493421,-.0328947){16}{\line(1,0){.0493421}}
\multiput(45.193,13.088)(.0493421,-.0328947){16}{\line(1,0){.0493421}}
\multiput(46.772,12.035)(.0493421,-.0328947){16}{\line(1,0){.0493421}}
\multiput(48.351,10.982)(.0493421,-.0328947){16}{\line(1,0){.0493421}}
\multiput(59.93,29.93)(.0493421,-.0328947){16}{\line(1,0){.0493421}}
\multiput(61.509,28.877)(.0493421,-.0328947){16}{\line(1,0){.0493421}}
\multiput(63.088,27.824)(.0493421,-.0328947){16}{\line(1,0){.0493421}}
\multiput(64.667,26.772)(.0493421,-.0328947){16}{\line(1,0){.0493421}}
\multiput(66.245,25.719)(.0493421,-.0328947){16}{\line(1,0){.0493421}}
\multiput(67.824,24.667)(.0493421,-.0328947){16}{\line(1,0){.0493421}}
\multiput(69.403,23.614)(.0493421,-.0328947){16}{\line(1,0){.0493421}}
\multiput(70.982,22.561)(.0493421,-.0328947){16}{\line(1,0){.0493421}}
\multiput(72.561,21.509)(.0493421,-.0328947){16}{\line(1,0){.0493421}}
\multiput(74.14,20.456)(.0493421,-.0328947){16}{\line(1,0){.0493421}}
\multiput(75.719,19.403)(.0493421,-.0328947){16}{\line(1,0){.0493421}}
\multiput(77.298,18.351)(.0493421,-.0328947){16}{\line(1,0){.0493421}}
\multiput(78.877,17.298)(.0493421,-.0328947){16}{\line(1,0){.0493421}}
\multiput(80.456,16.245)(.0493421,-.0328947){16}{\line(1,0){.0493421}}
\multiput(82.035,15.193)(.0493421,-.0328947){16}{\line(1,0){.0493421}}
\multiput(83.614,14.14)(.0493421,-.0328947){16}{\line(1,0){.0493421}}
\multiput(85.193,13.088)(.0493421,-.0328947){16}{\line(1,0){.0493421}}
\multiput(86.772,12.035)(.0493421,-.0328947){16}{\line(1,0){.0493421}}
\multiput(88.351,10.982)(.0493421,-.0328947){16}{\line(1,0){.0493421}}
\multiput(109.93,9.93)(.0493421,.0328947){16}{\line(1,0){.0493421}}
\multiput(111.509,10.982)(.0493421,.0328947){16}{\line(1,0){.0493421}}
\multiput(113.088,12.035)(.0493421,.0328947){16}{\line(1,0){.0493421}}
\multiput(114.667,13.088)(.0493421,.0328947){16}{\line(1,0){.0493421}}
\multiput(116.245,14.14)(.0493421,.0328947){16}{\line(1,0){.0493421}}
\multiput(117.824,15.193)(.0493421,.0328947){16}{\line(1,0){.0493421}}
\multiput(119.403,16.245)(.0493421,.0328947){16}{\line(1,0){.0493421}}
\multiput(120.982,17.298)(.0493421,.0328947){16}{\line(1,0){.0493421}}
\multiput(122.561,18.351)(.0493421,.0328947){16}{\line(1,0){.0493421}}
\multiput(124.14,19.403)(.0493421,.0328947){16}{\line(1,0){.0493421}}
\multiput(125.719,20.456)(.0493421,.0328947){16}{\line(1,0){.0493421}}
\multiput(127.298,21.509)(.0493421,.0328947){16}{\line(1,0){.0493421}}
\multiput(128.877,22.561)(.0493421,.0328947){16}{\line(1,0){.0493421}}
\multiput(130.456,23.614)(.0493421,.0328947){16}{\line(1,0){.0493421}}
\multiput(132.035,24.667)(.0493421,.0328947){16}{\line(1,0){.0493421}}
\multiput(133.614,25.719)(.0493421,.0328947){16}{\line(1,0){.0493421}}
\multiput(135.193,26.772)(.0493421,.0328947){16}{\line(1,0){.0493421}}
\multiput(136.772,27.824)(.0493421,.0328947){16}{\line(1,0){.0493421}}
\multiput(138.351,28.877)(.0493421,.0328947){16}{\line(1,0){.0493421}}
\put(149.93,9.93){\line(1,0){.9677}}
\put(151.865,9.93){\line(1,0){.9677}}
\put(153.801,9.93){\line(1,0){.9677}}
\put(155.736,9.93){\line(1,0){.9677}}
\put(157.672,9.93){\line(1,0){.9677}}
\put(159.607,9.93){\line(1,0){.9677}}
\put(161.543,9.93){\line(1,0){.9677}}
\put(163.478,9.93){\line(1,0){.9677}}
\put(165.414,9.93){\line(1,0){.9677}}
\put(167.349,9.93){\line(1,0){.9677}}
\put(169.285,9.93){\line(1,0){.9677}}
\put(171.22,9.93){\line(1,0){.9677}}
\put(173.156,9.93){\line(1,0){.9677}}
\put(175.091,9.93){\line(1,0){.9677}}
\put(177.026,9.93){\line(1,0){.9677}}
\put(178.962,9.93){\line(1,0){.9677}}
\multiput(159.93,49.93)(.0320513,-.0641026){13}{\line(0,-1){.0641026}}
\multiput(160.763,48.263)(.0320513,-.0641026){13}{\line(0,-1){.0641026}}
\multiput(161.596,46.596)(.0320513,-.0641026){13}{\line(0,-1){.0641026}}
\multiput(162.43,44.93)(.0320513,-.0641026){13}{\line(0,-1){.0641026}}
\multiput(163.263,43.263)(.0320513,-.0641026){13}{\line(0,-1){.0641026}}
\multiput(164.096,41.596)(.0320513,-.0641026){13}{\line(0,-1){.0641026}}
\multiput(164.93,39.93)(.0320513,-.0641026){13}{\line(0,-1){.0641026}}
\multiput(165.763,38.263)(.0320513,-.0641026){13}{\line(0,-1){.0641026}}
\multiput(166.596,36.596)(.0320513,-.0641026){13}{\line(0,-1){.0641026}}
\multiput(167.43,34.93)(.0320513,-.0641026){13}{\line(0,-1){.0641026}}
\multiput(168.263,33.263)(.0320513,-.0641026){13}{\line(0,-1){.0641026}}
\multiput(169.096,31.596)(.0320513,-.0641026){13}{\line(0,-1){.0641026}}
\multiput(199.93,49.93)(-.0320513,-.0641026){13}{\line(0,-1){.0641026}}
\multiput(199.096,48.263)(-.0320513,-.0641026){13}{\line(0,-1){.0641026}}
\multiput(198.263,46.596)(-.0320513,-.0641026){13}{\line(0,-1){.0641026}}
\multiput(197.43,44.93)(-.0320513,-.0641026){13}{\line(0,-1){.0641026}}
\multiput(196.596,43.263)(-.0320513,-.0641026){13}{\line(0,-1){.0641026}}
\multiput(195.763,41.596)(-.0320513,-.0641026){13}{\line(0,-1){.0641026}}
\multiput(194.93,39.93)(-.0320513,-.0641026){13}{\line(0,-1){.0641026}}
\multiput(194.096,38.263)(-.0320513,-.0641026){13}{\line(0,-1){.0641026}}
\multiput(193.263,36.596)(-.0320513,-.0641026){13}{\line(0,-1){.0641026}}
\multiput(192.43,34.93)(-.0320513,-.0641026){13}{\line(0,-1){.0641026}}
\multiput(191.596,33.263)(-.0320513,-.0641026){13}{\line(0,-1){.0641026}}
\multiput(190.763,31.596)(-.0320513,-.0641026){13}{\line(0,-1){.0641026}}
\multiput(99.93,29.93)(-.0493421,-.0328947){16}{\line(-1,0){.0493421}}
\multiput(98.351,28.877)(-.0493421,-.0328947){16}{\line(-1,0){.0493421}}
\multiput(96.772,27.824)(-.0493421,-.0328947){16}{\line(-1,0){.0493421}}
\multiput(95.193,26.772)(-.0493421,-.0328947){16}{\line(-1,0){.0493421}}
\multiput(93.614,25.719)(-.0493421,-.0328947){16}{\line(-1,0){.0493421}}
\multiput(92.035,24.667)(-.0493421,-.0328947){16}{\line(-1,0){.0493421}}
\multiput(90.456,23.614)(-.0493421,-.0328947){16}{\line(-1,0){.0493421}}
\multiput(88.877,22.561)(-.0493421,-.0328947){16}{\line(-1,0){.0493421}}
\multiput(87.298,21.509)(-.0493421,-.0328947){16}{\line(-1,0){.0493421}}
\multiput(85.719,20.456)(-.0493421,-.0328947){16}{\line(-1,0){.0493421}}
\multiput(84.14,19.403)(-.0493421,-.0328947){16}{\line(-1,0){.0493421}}
\multiput(82.561,18.351)(-.0493421,-.0328947){16}{\line(-1,0){.0493421}}
\multiput(80.982,17.298)(-.0493421,-.0328947){16}{\line(-1,0){.0493421}}
\multiput(79.403,16.245)(-.0493421,-.0328947){16}{\line(-1,0){.0493421}}
\multiput(77.824,15.193)(-.0493421,-.0328947){16}{\line(-1,0){.0493421}}
\multiput(76.245,14.14)(-.0493421,-.0328947){16}{\line(-1,0){.0493421}}
\multiput(74.667,13.088)(-.0493421,-.0328947){16}{\line(-1,0){.0493421}}
\multiput(73.088,12.035)(-.0493421,-.0328947){16}{\line(-1,0){.0493421}}
\multiput(71.509,10.982)(-.0493421,-.0328947){16}{\line(-1,0){.0493421}}
\end{picture}
\end{center}
\caption{Two graphs in ${\cal G}$.
The left graph contains one copy of $K_4^*$.
Note that only the copies of $K_3$ contain unmarked vertices of degree less than $3$, which implies that all additional edges are incident with these vertices.}\label{fig2}
\end{figure}

\begin{theorem}\label{theorem3}
A connected subcubic graph $G$ satisfies $2\alpha(G)={\rm diss}(G)$ if and only if it belongs to ${\cal G}$.
\end{theorem}
Having considered the influence of degree conditions on (\ref{e0-1}),
we now add further structural conditions.
For bipartite graphs, a simple argument yields the following.

\begin{proposition}\label{proposition1}
If $G$ is a connected bipartite graph of maximum degree at most $\Delta$, then
\begin{eqnarray*}\label{e0-2}
\alpha(G)\geq \frac{1}{2}\left(1+\frac{1}{2(\Delta-1)}\right){\rm diss}(G)-\frac{1}{2(\Delta-1)}.
\end{eqnarray*}
\end{proposition}
Proposition \ref{proposition1} is best possible,
and its proof implies that all extremal graphs are trees.
Generalizing the subcubic extremal tree shown in Figure \ref{fig1} 
easily allows to construct arbitrarily large extremal trees of maximum degree $\Delta$
for all values of $\Delta\geq 2$. 

\begin{figure}[H]
\begin{center}
\unitlength 0.4mm 
\linethickness{0.4pt}
\ifx\plotpoint\undefined\newsavebox{\plotpoint}\fi 
\begin{picture}(246,66)(0,0)
\put(5,25){\circle*{3}}
\put(65,25){\circle*{3}}
\put(145,25){\circle*{3}}
\put(225,25){\circle*{3}}
\put(5,45){\circle*{3}}
\put(65,45){\circle*{3}}
\put(145,45){\circle*{3}}
\put(225,45){\circle*{3}}
\put(25,25){\circle*{3}}
\put(85,25){\circle*{3}}
\put(165,25){\circle*{3}}
\put(245,25){\circle*{3}}
\put(25,45){\circle*{3}}
\put(85,45){\circle*{3}}
\put(165,45){\circle*{3}}
\put(245,45){\circle*{3}}
\put(25,65){\circle*{3}}
\put(45,25){\circle*{3}}
\put(105,25){\circle*{3}}
\put(185,25){\circle*{3}}
\put(125,25){\circle*{3}}
\put(205,25){\circle*{3}}
\put(45,45){\circle*{3}}
\put(105,45){\circle*{3}}
\put(185,45){\circle*{3}}
\put(125,45){\circle*{3}}
\put(205,45){\circle*{3}}
\put(5,25){\line(0,1){20}}
\put(65,25){\line(0,1){20}}
\put(145,25){\line(0,1){20}}
\put(225,25){\line(0,1){20}}
\put(5,45){\line(1,1){20}}
\put(25,65){\line(0,-1){20}}
\put(25,45){\line(0,-1){20}}
\put(85,45){\line(0,-1){20}}
\put(165,45){\line(0,-1){20}}
\put(245,45){\line(0,-1){20}}
\put(25,65){\line(1,-1){20}}
\put(45,45){\line(0,-1){20}}
\put(105,45){\line(0,-1){20}}
\put(185,45){\line(0,-1){20}}
\put(125,45){\line(0,-1){20}}
\put(205,45){\line(0,-1){20}}
\put(65,5){\circle*{3}}
\put(145,5){\circle*{3}}
\put(225,5){\circle*{3}}
\put(105,65){\circle*{3}}
\put(185,65){\circle*{3}}
\put(45,25){\line(1,-1){20}}
\put(125,25){\line(1,-1){20}}
\put(205,25){\line(1,-1){20}}
\put(65,5){\line(0,1){20}}
\put(145,5){\line(0,1){20}}
\put(225,5){\line(0,1){20}}
\put(85,25){\line(-1,-1){20}}
\put(165,25){\line(-1,-1){20}}
\put(245,25){\line(-1,-1){20}}
\put(105,65){\line(-1,-1){20}}
\put(185,65){\line(-1,-1){20}}
\put(105,65){\line(0,-1){20}}
\put(185,65){\line(0,-1){20}}
\put(105,65){\line(1,-1){20}}
\put(185,65){\line(1,-1){20}}
\end{picture}
\end{center}
\caption{A subcubic tree $T$ with $\alpha(T)=\frac{5}{8}{\rm diss}(T)-\frac{1}{4}$.
}\label{fig1}
\end{figure}
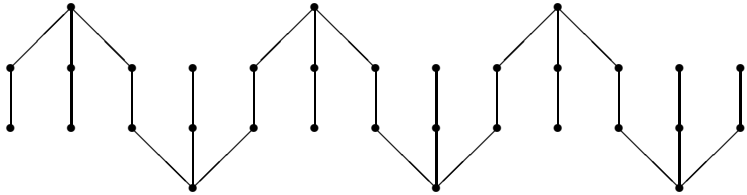
For regular triangle-free graphs, we obtain the following result.

\begin{theorem}\label{theorem2}
If $G$ is a triangle-free $\Delta$-regular graph for some $\Delta\geq 3$, then 
$$\alpha(G)\geq 
\frac{1}{2}\left(1+\frac{(\Delta-1)(\Delta+1)}{2^{\Delta}\Delta^2+(\Delta-1)(\Delta+1)}\right){\rm diss}(G)
=\frac{1}{2}\left(1+\Omega\left(\frac{1}{2^\Delta}\right)\right){\rm diss}(G).$$
\end{theorem}
For $\Delta=3$, that is, for a triangle-free cubic graph $G$, 
Theorem \ref{theorem2} implies
$\alpha(G)\geq \frac{11}{20}{\rm diss}(G)$.
As pointed out after the proof of Theorem \ref{theorem1},
its proof allows to show $\alpha(G)\geq \frac{5}{8}{\rm diss}(G)$ 
for a triangle-free cubic graph $G$.
We believe that $\frac{5}{8}$ is the right factor 
also for triangle-free subcubic graphs.

\begin{conjecture}\label{conjecture2}
If $G$ is a connected triangle-free subcubic graph, 
then $\alpha(G)\geq \frac{5}{8}{\rm diss}(G)-\frac{1}{4}$.
\end{conjecture}

\section{Proofs}\label{section2}

We begin with the construction of 
$r$-regular extremal graphs for (\ref{e0-1}) for $r\geq 4$:
Let $k$ and $\ell$ be positive integers at least $2$.
Let the graph $G(k,\ell)$ arise from the disjoint union of $\ell$ copies $K(1),\ldots,K(\ell)$ of $K_{2k}$,
where we partition the vertex set of each $K(i)$ into two sets $L(i)$ and $R(i)$ of order $k$ for every $i$,
by adding a matching $M(i)$ of size $k$ between $R(i)$ and $L(i+1)$ for every $i$,
where we consider $i$ modulo $\ell$.
Since $G(k,\ell)$ is covered by $\ell$ cliques, 
we have $\alpha(G(k,\ell))\leq \ell$ and ${\rm diss}(G(k,\ell))\leq 2\ell$.
Since a set containing two vertices from $L(i)$ for every $i$ is a dissociation set, 
we have $2\alpha(G(k,\ell))\geq {\rm diss}(G(k,\ell))\geq 2\ell$.
Altogether, it follows that $2\alpha(G(k,\ell))={\rm diss}(G(k,\ell))$.
By construction, the graph $G(k,\ell)$ is $2k$-regular.
Let the graph $G'(k,\ell)$ arise from $G(k,\ell)$
by adding a matching of size $k$ that is disjoint from $M(i)$ between $R(i)$ and $L(i+1)$ for every $i$,
where we consider $i$ modulo $\ell$.
Arguing as above, it follows that $G'(k,\ell)$ is $(2k+1)$-regular
and satisfies $2\alpha(G'(k,\ell))={\rm diss}(G'(k,\ell))$.
For the sake of completeness, let us point out that 
$\alpha(G)={\rm diss}(G)$ if $G$ is $0$-regular,
$2\alpha(G)={\rm diss}(G)$ if $G$ is $1$-regular, 
$\alpha(C_n)=\left\lfloor\frac{n}{2}\right\rfloor$, and 
${\rm diss}(C_n)=\left\lfloor\frac{2n}{3}\right\rfloor$.

Our next goal is the proof of Theorem \ref{theorem1},
which uses the following simple lemma.

\begin{lemma}\label{lemmaDiss}
If $G$ is a subcubic graph and 
$D$ is a maximum dissociation set of $G$ 
that maximizes the number $p$ of isolated vertices in $G[D]$, 
then $G[V(G) \setminus D]$ 
has maximum degree at most $1$.
\end{lemma}
\begin{proof}
Suppose, for a contradiction, 
that some vertex $u$ in $R=V(G) \setminus D$ has two neighbors in $R$.
The choice of $D$ implies that $u$ has a neighbor $v$ in $D$.
Since $G$ is subcubic, the vertex $v$ is the only neighbor of $u$ in $D$,
and the choice of $D$ further implies that $v$ has a unique neighbor $w$ in $D$.
Now, replacing $v$ by $u$ within $D$
yields a maximum dissociation set $\bar{D}$ 
with more than $p$ isolated vertices in $G[\bar{D}]$,
contradicting the choice of $D$, and completing the proof.
\end{proof}
The following proof of Theorem \ref{theorem1} is set up in a way
that allows to extract several structural features of the extremal graphs.

\begin{proof}[Proof of Theorem \ref{theorem1}]
Let $G$ be a connected cubic graph of order $n$ at least $6$.
Let $D$ be a maximum dissociation set of $G$ 
that maximizes the number $p$ of isolated vertices in $G[D]$.
Let $G[D]$ contain $q$ components of order $2$.
By Lemma \ref{lemmaDiss}, 
the set $\bar{D} = V(G) \setminus D$ is also a dissociation set.
Let $G[\bar{D}]$ contain $r$ isolated vertices and $s$ components of order $2$,
in particular, we have $n=p+2q+r+2s$. 
Counting the edges between $D$ and $\bar{D}$, we obtain $3p+4q = 3r+4s$.
By the choice of $G$ and the structure of $G[\bar{D}]$, 
we obtain 
${\rm diss}(G)=p +2q$
and
$\alpha(G)\geq r+s$. 
Now,
\begin{eqnarray}\label{ethm1}
n &=& p+2q+r+2s
=3(p + 2q)+ p-2(r+s)
\geq 3{\rm diss}(G)+p-2\alpha(G).
\end{eqnarray}
Now, suppose that $\frac{\alpha(G)}{{\rm diss}(G)} \leq \frac{3}{5}$.
By Brooks' theorem \cite{br}, we have $\alpha (G)\geq n/3$, 
and (\ref{ethm1}) implies
$$3\alpha (G)\geq n\geq 3{\rm diss}(G)+p-2\alpha(G)
\geq 3\alpha (G),$$
where the last inequality uses $\frac{\alpha(G)}{{\rm diss}(G)} \leq \frac{3}{5}$.
It follows that equality holds throughout this inequality chain. 
This implies 
$\alpha(G) =n/3$, 
$p = 0$, 
$r+s = \alpha(G)$, and 
$2q = {\rm diss}(G) = \frac{5}{3}\alpha(G) = \frac{5}{9}n$.
Furthermore, using $n = p+2q+r+s+s$, 
we obtain $s= \frac{n}{9}$ and $r=\frac{2n}{9}$.
\end{proof}
Theorem \ref{conjecture1} concerning the extremal graphs for Theorem \ref{theorem1}
requires quite some technical preparation, 
which is why we postpone it to the end of this section.
If $G$ is a triangle-free cubic graph of order $n$,
then $\alpha(G)\geq \frac{5}{14}n$ \cite{heth,st}.
In this case
(\ref{ethm1}) implies 
$\frac{14}{5}\alpha(G)\geq 3{\rm diss}(G)-2\alpha(G)$
and, hence,
$$\alpha(G)\geq \frac{5}{8}{\rm diss}(G).$$
We proceed to the proof of the characterization 
of the extremal subcubic graphs for (\ref{e0-1}).

\begin{proof}[Proof of Theorem \ref{theorem3}]
Let $G$ belong to ${\cal G}$.
If $G=K_4$, then $2\alpha(G)={\rm diss}(G)=2$.
Now, let $G\not=K_4$.
Let $G$ arise from 
$n_2$ copies of $K_2$,
$n_3$ copies of $K_3$, and
$n_4$ copies of $K_4^*$
by adding additional edges, each incident with at most one marked vertex.
Since $D(G)$ is a dissociation set by construction, 
we obtain $2\alpha(G)\geq {\rm diss}(G)\geq |D(G)|=2n_2+2n_3+4n_4$.
Since every independent set in $G$ contains at most one vertex from each $K_2$ and $K_3$ 
as well as at most $2$ vertices from each $K_4^*$,
we obtain $\alpha(G)\leq n_2+n_3+2n_4$.
It follows that $2\alpha(G)={\rm diss}(G)$;
in particular, the set $D(G)$ is a maximum dissociation set.

For the converse, let $G$ be a connected subcubic graph with $2\alpha(G)={\rm diss}(G)$
that is distinct from $K_4$.
Let $D$ be a maximum dissociation set in $G$.
Since $2\alpha(G)={\rm diss}(G)$,
the graph $G[D]$ consists of $\alpha(G)$ copies of $K_2$.
If some vertex $u$ in $R=V(G)\setminus D$ has at most one neighbor in each $K_2$ in $G[D]$,
then $G$ has an independent set containing $u$ as well as one vertex from each $K_2$ in $G[D]$,
which implies the contradiction $\alpha(G)>{\rm diss}(G)/2$.
Hence, for every vertex $u$ in $R$, there is a $K_2$ component in $G[D]$, say $K(u)$, 
such that $u$ is adjacent to both vertices of $K(u)$.
Since $G$ is subcubic, the $K_2$ component $K(u)$ is uniquely determined for every $u$ in $R$,
and, for every $K_2$ component $K$ in $G[D]$, there are at most two vertices $u$ in $R$ with $K=K(u)$.

Suppose that $K(u)=K(v)$ for two distinct vertices $u$ and $v$ in $R$.
Since $G$ is not $K_4$, the vertices $u$ and $v$ are not adjacent.
If every $K_2$ component of $G[D]$ that is distinct from $K(u)$ 
contains a vertex that is adjacent neither to $u$ nor to $v$,
then $G$ has an independent set containing $u$ and $v$ as well as one vertex from each $K_2$ in $G[D]$
that is distinct from $K(u)$, which implies the contradiction $\alpha(G)>{\rm diss}(G)/2$.
Since $G$ is subcubic, 
it follows that there is a unique $K_2$ component in $G[D]$ that is distinct from $K(u)$, say $K(\{ u,v\})$,
such that $u$ is adjacent to one vertex of $K(\{ u,v\})$ and $v$ is adjacent to the other vertex of $K(\{ u,v\})$.
If there is some vertex $w$ in $R$ with $K(w)=K(\{ u,v\})$, 
then $G$ has an independent set containing $u$, $v$, and $w$ 
as well as one vertex from each $K_2$ in $G[D]$ that is distinct from $K(u)$ and $K(w)$, 
which implies the contradiction $\alpha(G)>{\rm diss}(G)/2$.
Hence, there is no vertex $w$ in $R$ with $K(w)=K(\{ u,v\})$.
If there are two distinct vertices $u'$ and $v'$ in $R\setminus \{ u,v\}$ with $K(\{ u,v\})=K(\{ u',v'\})$,
then $G$ has an independent set containing $u$, $v$, $u'$, and $v'$ 
as well as one vertex from each $K_2$ in $G[D]$ that is distinct from $K(u)$, $K(u')$, and $K(\{ u,v\})$, 
which implies the contradiction $\alpha(G)>{\rm diss}(G)/2$.
Hence, there are no two distinct vertices $u'$ and $v'$ in $R\setminus \{ u,v\}$ with $K(\{ u,v\})=K(\{ u',v'\})$.
Note that $\{ u,v\}\cup K(u)\cup K(\{ u,v\})$ induces a $K_4^*$ in $G$.

We are now in a position to describe a decomposition of $G$ into disjoint copies of $K_2$, $K_3$, and $K_4^*$
as in the definition of ${\cal G}$:
\begin{itemize}
\item For every $K_2$ component $K$ in $G[D]$ with $K=K(u)=K(v)$ for two distinct vertices $u$ and $v$ in $R$,
the set $\{ u,v\}\cup V(K)\cup K(\{ u,v\})$ induces a $K_4^*$ in $G$,
which will be part of the decomposition.
\item For every $K_2$ component $K$ in $G[D]$ with $K=K(u)$ for exactly one vertex $u$ in $R$,
the set $\{ u\}\cup V(K)$ induces a $K_3$ in $G$,
which will be part of the decomposition.
\item For every $K_2$ component $K$ in $G[D]$ 
such that there is no vertex $u$ in $R$ with $K=K(u)$
and there are no two distinct vertices $u$ and $v$ in $R$ with $K=K(\{ u,v\})$,
the set $V(K)$ induces a $K_2$ in $G$,
which will be part of the decomposition.
\end{itemize}
The above observations imply that all described copies of $K_2$, $K_3$, and $K_4^*$ are disjoint.
Furthermore, since $D$ is a dissociation set, all additional edges of $G$ are incident with at most one vertex in $D$.
Since the vertices in $D$ correspond to the marked vertices in the definition of ${\cal G}$,
we obtain that $G$ belongs to ${\cal G}$,
which completes the proof.
\end{proof}
Our best possible result for bipartite graphs has quite a simple proof.

\begin{proof}[Proof of Proposition \ref{proposition1}]
Let $n$ be the order of $G$.
Let $D$ be a maximum dissociation set in $G$.
Let $G[D]$ contain $p$ isolated vertices and $q$ components of order $2$,
that is, ${\rm diss}(G)=p+2q$.
The number $m$ of edges incident with the vertices in $R=V(G)\setminus D$ satisfies $m\leq \Delta|R|=\Delta(n-p-2q)$.
Removing these $m$ edges results in a graph with at most $1+m$ components:
$p+(n-p-2q)$ isolated vertices and $q$ components of order $2$.
We obtain
\begin{eqnarray}\label{etree}
1+\Delta(n-p-2q)\geq 1+m\geq p+(n-p-2q)+q=n-q,
\end{eqnarray}
which implies
$$\alpha(G)\geq \frac{n}{2}\stackrel{(\ref{etree})}{\geq} \frac{\Delta}{2(\Delta-1)}p+\frac{2\Delta-1}{4(\Delta-1)}2q-\frac{1}{2(\Delta-1)}
\geq \frac{2\Delta-1}{4(\Delta-1)}{\rm diss}(G)-\frac{1}{2(\Delta-1)}.$$
\end{proof}

The following proof of our main result about triangle-free graphs
is based on a probabilistic argument,
and recycles ideas from the probabilistic folklore proof
of the Caro-Wei bound on the independence number.

\begin{proof}[Proof of Theorem \ref{theorem2}]
Let $D$ be a maximum dissociation set in $G$.
Let $D$ contain $p$ isolated vertices and $q$ components of order $2$.
For $i\in \{ 0,1\}$, 
let $D_i$ be the set of vertices of degree $i$ in $G[D]$.
For $i\in \{ 0,1,\ldots,\Delta\}$,
let $R_i$ be the set of vertices in $V(G)\setminus D$ 
that have exactly $i$ neighbors in $D_1$,
and let $r_i=|R_i|$.
Since $G$ is $\Delta$-regular, we have
\begin{eqnarray}\label{ethm21}
\sum\limits_{i=1}^\Delta ir_i&=&2(\Delta-1)q.
\end{eqnarray}
Adding to $D_0$ one vertex from each of the $q$ components of $G[D_1]$
yields an independent set in $G$, which implies
\begin{eqnarray}\label{ethm22}
\alpha(G) &\geq & p+q={\rm diss}(G)-q.
\end{eqnarray}
We now construct another independent set by the following random procedure:
For every $K_2$ component in $G[D_1]$,
we select independently one of its two vertices with probability $1/2$.
Let $I_1$ be the set of selected vertices.
Let $\pi$ be a linear ordering of the vertices in $V(G)\setminus D_1$
chosen uniformly at random.
Let $I_2$ be the set of all vertices $u$ in $V(G)\setminus D_1$ such that
\begin{itemize}
\item if $u\in D_0$, then $u$ appears within $\pi$ before all its neighbors, and
\item if $u\in V(G)\setminus D$, then 
$u$ appears within $\pi$ before all its neighbors within $V(G)\setminus D_1$,
and $u$ has no neighbor in $I_1$.
\end{itemize}
Note that
$$\mathbb{P}[u\in I_2]=
\begin{cases}
\frac{1}{\Delta+1}, & \mbox{ if $u\in D_0$, and}\\
\frac{1}{2^i(\Delta-i+1)}, & \mbox{ if $u\in R_i$ for some $i\in \{ 0,1,\ldots,\Delta\}$}.
\end{cases}$$
By linearity of expectation,
$$\mathbb{E}[|I_2|]
=\frac{p+r_0}{\Delta+1}+\sum_{i=1}^\Delta \frac{r_i}{2^i(\Delta-i+1)}
=\frac{p+r_0}{\Delta+1}+\sum_{i=1}^\Delta \frac{ir_i}{2^i i (\Delta-i+1)}.$$ 
Since 
$$\max\big\{2^i i (\Delta-i+1): i\in \{ 1,\ldots,\Delta\}\big\}=2^\Delta\Delta,$$
we obtain, using (\ref{ethm21}), that 
$$\mathbb{E}[|I_2|]\geq 
\frac{p}{\Delta+1}+\frac{2(\Delta-1)q}{2^\Delta \Delta}.$$
Altogether, considering the independent set $I_1\cup I_2$,
and using the first moment method, we obtain
\begin{eqnarray}\label{ethm23}
\alpha(G) &\geq & q+\frac{p}{\Delta+1}+\frac{2(\Delta-1)q}{2^\Delta \Delta}
=\frac{1}{\Delta+1}{\rm diss}(G)+\left(1+\frac{2(\Delta-1)}{2^\Delta \Delta}-\frac{2}{\Delta+1}\right)q.
\end{eqnarray}
Equating the two lower bounds (\ref{ethm22}) and (\ref{ethm23}) yields 
$$q=\frac{2^{\Delta-1}\Delta^2}{2^{\Delta}\Delta^2+(\Delta-1)(\Delta+1)}{\rm diss}(G)
=\frac{1}{2}\left(1-\frac{(\Delta-1)(\Delta+1)}{2^{\Delta}\Delta^2+(\Delta-1)(\Delta+1)}\right){\rm diss}(G).$$
Since the two bounds (\ref{ethm22}) and (\ref{ethm23}) are decreasing and increasing in $q$, respectively, we obtain
\begin{eqnarray*}
\alpha(G) & \geq & {\rm diss}(G)-\frac{1}{2}\left(1-\frac{(\Delta-1)(\Delta+1)}{2^{\Delta}\Delta^2+(\Delta-1)(\Delta+1)}\right){\rm diss}(G)\\
& = & \frac{1}{2}\left(1+\frac{(\Delta-1)(\Delta+1)}{2^{\Delta}\Delta^2+(\Delta-1)(\Delta+1)}\right){\rm diss}(G),
\end{eqnarray*}
which completes the proof.
\end{proof}
We proceed to Theorem \ref{conjecture1},
which describes the structure of the extremal graphs for Theorem \ref{theorem1}.

Let $k$ be a positive integer.
Let ${\cal H}_k$ be the set of all connected subcubic multigraphs $H$ 
with the following properties:
\begin{itemize} 
\item $H$ may contain parallel edges but no loops.
\item $3n_1+2n_2+n_3=6k$, where $n_i$ is the number of vertices of $H$ of degree $i$,
and the degree of a vertex in $H$ is the number of incident edges.
\item $H$ has an induced matching $M$ of size $k$ that covers all vertices of degree $1$.
\item $H-M$ has an orientation $\overrightarrow{H-M}$ such that every vertex $u$ of $H$
that is not incident with an edge in $M$ has exactly two outgoing edges.
\end{itemize}
${\cal H}_k$ is quite a rich family of multigraphs; 
starting with an integral solution of $3n_1+2n_2+n_3=6k$,
the construction of elements of ${\cal H}_k$ is simple.
 
Let ${\cal G}_k$ be the set of all graphs $G$ that arise from a graph $H$ in ${\cal H}_k$ by 
\begin{itemize} 
\item replacing every vertex of $H$ of degree $3$ with a triangle as illustrated in Figure \ref{figd3},

\begin{figure}[H]
\begin{center}
\unitlength 0.6mm 
\linethickness{0.4pt}
\ifx\plotpoint\undefined\newsavebox{\plotpoint}\fi 
\begin{picture}(98,43)(0,0)
\put(64,11){\circle*{2}}
\put(84,11){\circle*{2}}
\put(74,26){\circle*{2}}
\put(74,26){\line(-2,-3){10}}
\put(64,11){\line(1,0){20}}
\put(84,11){\line(-2,3){10}}
\put(14,21){\circle*{2}}
\multiput(0,10)(.0428134557,.0336391437){327}{\line(1,0){.0428134557}}
\multiput(50,0)(.0428134557,.0336391437){327}{\line(1,0){.0428134557}}
\multiput(28,10)(-.0428134557,.0336391437){327}{\line(-1,0){.0428134557}}
\multiput(98,0)(-.0428134557,.0336391437){327}{\line(-1,0){.0428134557}}
\put(14,21){\line(0,1){17}}
\put(74,26){\line(0,1){17}}
\put(34,21){\vector(1,0){20}}
\end{picture}
\end{center}
\caption{Replacing a vertex of degree $3$.}\label{figd3}
\end{figure}
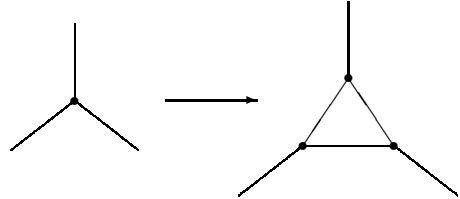

\item replacing every vertex of $H$ of degree $2$ with a $K_4^*$ as illustrated in Figure \ref{figd2}, and

\begin{figure}[H]
\begin{center}
\unitlength 0.6mm 
\linethickness{0.4pt}
\ifx\plotpoint\undefined\newsavebox{\plotpoint}\fi 
\begin{picture}(125,42)(0,0)
\put(95,0){\circle*{2}}
\put(95,20){\circle*{2}}
\put(95,40){\circle*{2}}
\put(115,0){\circle*{2}}
\put(115,20){\circle*{2}}
\put(115,40){\circle*{2}}
\put(95,0){\line(1,0){20}}
\put(95,40){\line(1,0){20}}
\put(15,20){\line(1,0){20}}
\put(95,40){\line(0,-1){40}}
\put(95,0){\line(1,1){20}}
\put(115,20){\line(0,-1){20}}
\put(115,0){\line(-1,1){20}}
\put(115,40){\line(0,-1){20}}
\put(115,40){\line(1,0){10}}
\put(35,20){\line(1,0){10}}
\put(95,40){\line(-1,0){10}}
\put(15,20){\line(-1,0){10}}
\put(25,20){\circle*{2}}
\put(55,20){\vector(1,0){20}}
\put(95,0){\circle{4}}
\put(115,0){\circle{4}}
\end{picture}
\end{center}
\caption{Replacing a vertex of degree $2$.}\label{figd2}
\end{figure}
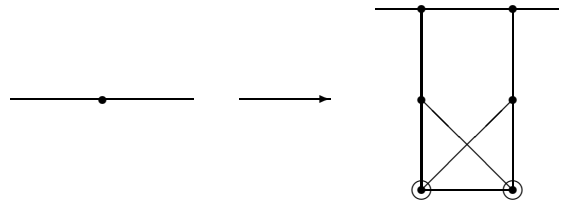

\item replacing every vertex of $H$ of degree $1$ with one of the two graphs of order $9$
as illustrated in Figure \ref{figd1}.

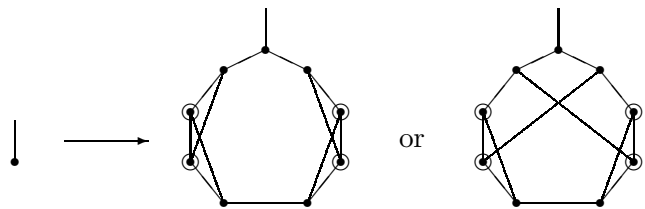
\begin{figure}[H]
\begin{center}
\unitlength .55mm 
\linethickness{0.4pt}
\ifx\plotpoint\undefined\newsavebox{\plotpoint}\fi 
\begin{picture}(150,47)(0,0)
\put(50,0){\circle*{2}}
\put(120,0){\circle*{2}}
\put(70,0){\circle*{2}}
\put(140,0){\circle*{2}}
\put(50,0){\line(1,0){20}}
\put(120,0){\line(1,0){20}}
\put(42,10){\circle*{2}}
\put(112,10){\circle*{2}}
\put(78,10){\circle*{2}}
\put(148,10){\circle*{2}}
\put(42,22){\circle*{2}}
\put(112,22){\circle*{2}}
\put(78,22){\circle*{2}}
\put(148,22){\circle*{2}}
\put(50,32){\circle*{2}}
\put(120,32){\circle*{2}}
\put(70,32){\circle*{2}}
\put(140,32){\circle*{2}}
\put(60,37){\circle*{2}}
\put(0,10){\circle*{2}}
\put(130,37){\circle*{2}}
\put(60,37){\line(0,1){10}}
\put(0,10){\line(0,1){10}}
\put(130,37){\line(0,1){10}}
\put(60,37){\line(-2,-1){10}}
\put(130,37){\line(-2,-1){10}}
\put(50,32){\line(-4,-5){8}}
\put(120,32){\line(-4,-5){8}}
\put(42,22){\line(0,-1){12}}
\put(112,22){\line(0,-1){12}}
\put(42,10){\line(4,-5){8}}
\put(112,10){\line(4,-5){8}}
\put(60,37){\line(2,-1){10}}
\put(130,37){\line(2,-1){10}}
\put(70,32){\line(4,-5){8}}
\put(140,32){\line(4,-5){8}}
\put(78,22){\line(0,-1){12}}
\put(148,22){\line(0,-1){12}}
\put(78,10){\line(-4,-5){8}}
\put(148,10){\line(-4,-5){8}}
\multiput(42,22)(.055944056,-.153846154){143}{\line(0,-1){.153846154}}
\multiput(112,22)(.055944056,-.153846154){143}{\line(0,-1){.153846154}}
\multiput(78,22)(-.055944056,-.153846154){143}{\line(0,-1){.153846154}}
\multiput(148,22)(-.055944056,-.153846154){143}{\line(0,-1){.153846154}}
\multiput(78,10)(-.055944056,.153846154){143}{\line(0,1){.153846154}}
\multiput(50,32)(-.055944056,-.153846154){143}{\line(0,-1){.153846154}}
\multiput(120,32)(.0714285714,-.056122449){392}{\line(1,0){.0714285714}}
\multiput(140,32)(-.0714285714,-.056122449){392}{\line(-1,0){.0714285714}}
\put(42,22){\circle{4}}
\put(112,22){\circle{4}}
\put(42,10){\circle{4}}
\put(112,10){\circle{4}}
\put(78,22){\circle{4}}
\put(148,22){\circle{4}}
\put(78,10){\circle{4}}
\put(148,10){\circle{4}}
\put(95,15){\makebox(0,0)[cc]{or}}
\put(12,15){\vector(1,0){20}}
\end{picture}
\end{center}
\caption{Replacing a vertex of degree $1$.}\label{figd1}
\end{figure}
\end{itemize} 
The graph $G$ in Figure \ref{fig3} belongs to ${\cal G}_1$;
the corresponding multigraph $H$ is obtained by contracting the six triangles,
the matching $M$ consists of the edge between the two central triangles,
and the orientation $\overrightarrow{H-M}$ 
is shown on the corresponding eight edges of $G$.
Another example is shown in Figure \ref{figl}.

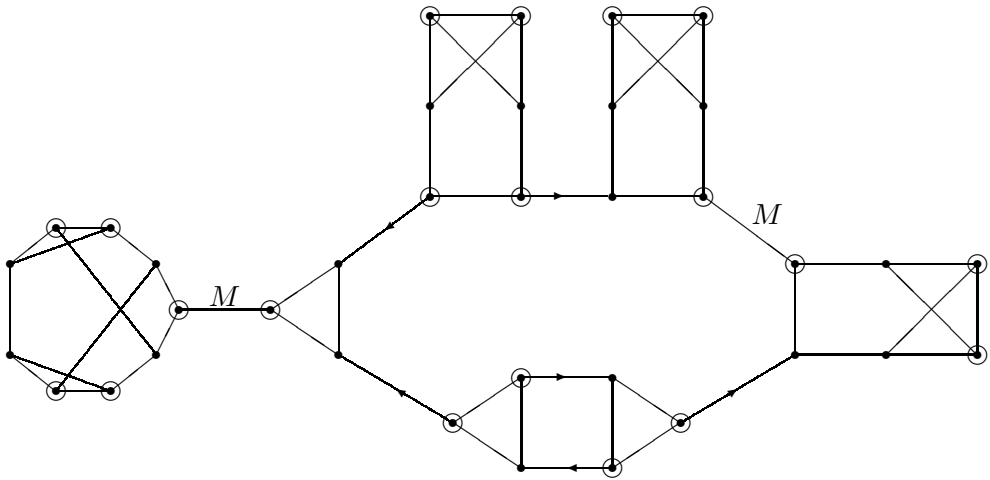
\begin{figure}[h]
\begin{center}
\unitlength .6mm 
\linethickness{0.4pt}
\ifx\plotpoint\undefined\newsavebox{\plotpoint}\fi 
\begin{picture}(214,102)(0,0)
\put(0,45){\circle*{2}}
\put(0,25){\circle*{2}}
\put(0,45){\line(0,-1){20}}
\put(10,53){\circle*{2}}
\put(10,17){\circle*{2}}
\put(22,53){\circle*{2}}
\put(22,17){\circle*{2}}
\put(32,45){\circle*{2}}
\put(32,25){\circle*{2}}
\put(37,35){\circle*{2}}
\put(37,35){\line(-1,2){5}}
\put(32,45){\line(-5,4){10}}
\put(22,53){\line(-1,0){12}}
\put(10,53){\line(-5,-4){10}}
\put(37,35){\line(-1,-2){5}}
\put(32,25){\line(-5,-4){10}}
\put(22,17){\line(-1,0){12}}
\put(10,17){\line(-5,4){10}}
\multiput(22,53)(-.153846154,-.055944056){143}{\line(-1,0){.153846154}}
\multiput(22,17)(-.153846154,.055944056){143}{\line(-1,0){.153846154}}
\multiput(32,45)(-.056122449,-.0714285714){392}{\line(0,-1){.0714285714}}
\multiput(32,25)(-.056122449,.0714285714){392}{\line(0,1){.0714285714}}
\put(22,53){\circle{4}}
\put(10,53){\circle{4}}
\put(22,17){\circle{4}}
\put(10,17){\circle{4}}
\put(212,25){\circle*{2}}
\put(92,100){\circle*{2}}
\put(132,100){\circle*{2}}
\put(192,25){\circle*{2}}
\put(92,80){\circle*{2}}
\put(132,80){\circle*{2}}
\put(172,25){\circle*{2}}
\put(92,60){\circle*{2}}
\put(132,60){\circle*{2}}
\put(212,45){\circle*{2}}
\put(112,100){\circle*{2}}
\put(152,100){\circle*{2}}
\put(192,45){\circle*{2}}
\put(112,80){\circle*{2}}
\put(152,80){\circle*{2}}
\put(172,45){\circle*{2}}
\put(112,60){\circle*{2}}
\put(152,60){\circle*{2}}
\put(212,25){\line(0,1){20}}
\put(92,100){\line(1,0){20}}
\put(132,100){\line(1,0){20}}
\put(172,25){\line(0,1){20}}
\put(92,60){\line(1,0){20}}
\put(132,60){\line(1,0){20}}
\put(172,25){\line(1,0){40}}
\put(92,60){\line(0,1){40}}
\put(132,60){\line(0,1){40}}
\put(212,25){\line(-1,1){20}}
\put(92,100){\line(1,-1){20}}
\put(132,100){\line(1,-1){20}}
\put(192,45){\line(1,0){20}}
\put(112,80){\line(0,1){20}}
\put(152,80){\line(0,1){20}}
\put(212,45){\line(-1,-1){20}}
\put(112,100){\line(-1,-1){20}}
\put(152,100){\line(-1,-1){20}}
\put(172,45){\line(1,0){20}}
\put(112,60){\line(0,1){20}}
\put(152,60){\line(0,1){20}}
\put(212,25){\circle{4}}
\put(92,100){\circle{4}}
\put(132,100){\circle{4}}
\put(212,45){\circle{4}}
\put(112,100){\circle{4}}
\put(152,100){\circle{4}}
\put(92,60){\circle{4}}
\put(172,45){\circle{4}}
\put(112,60){\circle{4}}
\put(152,60){\circle{4}}
\put(37,35){\circle{4}}
\put(72,25){\circle*{2}}
\put(112,0){\circle*{2}}
\put(132,0){\circle*{2}}
\put(72,45){\circle*{2}}
\put(112,20){\circle*{2}}
\put(132,20){\circle*{2}}
\put(57,35){\circle*{2}}
\put(97,10){\circle*{2}}
\put(147,10){\circle*{2}}
\put(57,35){\line(3,-2){15}}
\put(97,10){\line(3,-2){15}}
\put(147,10){\line(-3,-2){15}}
\put(72,25){\line(0,1){20}}
\put(112,0){\line(0,1){20}}
\put(132,0){\line(0,1){20}}
\put(72,45){\line(-3,-2){15}}
\put(112,20){\line(-3,-2){15}}
\put(132,20){\line(3,-2){15}}
\put(132,0){\circle{4}}
\put(112,20){\circle{4}}
\put(82,52.5){\vector(-4,-3){.117}}\multiput(92,60)(-.074906367,-.0561797753){267}{\line(-1,0){.074906367}}
\put(152,60){\line(4,-3){20}}
\put(121.5,60){\vector(1,0){.117}}\put(112,60){\line(1,0){19}}
\put(97,10){\circle{4}}
\put(147,10){\circle{4}}
\put(122,20){\vector(1,0){.117}}\put(112,20){\line(1,0){20}}
\put(122,0){\vector(-1,0){.117}}\put(132,0){\line(-1,0){20}}
\put(37,35){\line(1,0){20}}
\put(57,35){\circle{4}}
\put(159.5,17.5){\vector(3,2){.117}}\multiput(147,10)(.0936329588,.0561797753){267}{\line(1,0){.0936329588}}
\put(84.5,17.5){\vector(-3,2){.117}}\multiput(97,10)(-.0936329588,.0561797753){267}{\line(-1,0){.0936329588}}
\put(47,38){\makebox(0,0)[cc]{$M$}}
\put(166,56){\makebox(0,0)[cc]{$M$}}
\end{picture}
\end{center}
\caption{A graph in ${\cal G}_2$ with $n_1=1$, $n_2=3$, and $n_3=3$.}\label{figl}
\end{figure}

\begin{lemma}\label{lemma1}
For every positive integer $k$, every graph $G$ in ${\cal G}_k$
is connected, cubic, and satisfies $\alpha(G)=6k$ and ${\rm diss}(G)=10k$.
\end{lemma}
\begin{proof}
The construction of the graphs in ${\cal G}_k$ immediately implies 
that $G$ is connected and cubic.
Since $\alpha(K_3)=1$, $\alpha(K_4^*)=2$, and the independence number of 
the two graphs replacing vertices of degree $1$ shown in Figure \ref{figd1} is $3$,
we obtain $\alpha(G)\leq 3n_1+2n_2+n_3=6k$.
Now, we describe a dissociation set $D$ in $G$:
\begin{enumerate}[(i)]
\item The $k$ edges in $M$ correspond in an obvious way to $k$ edges in $G$,
and $D$ contains all $2k$ vertices incident with these edges;
cf.~the two encircled vertices in the two central triangles in Figure \ref{fig3}.
\item The edges in $E(H)\setminus M$ correspond in an obvious way to edges in $G$.
We orient these edges exactly as in $\overrightarrow{H-M}$,
and let $D$ contain the vertices of $G$ that have an outgoing oriented edge.
This yields $2(n_1+n_2+n_3-2k)$ further vertices in $D$,
cf.~the eight oriented edges in Figure \ref{fig3}.
\item For every vertex $u$ of degree $2$ in $H$, 
the set $D$ contains the two vertices of the corresponding copy of $K_4^*$
encircled in Figure \ref{figd2}.
This yields $2n_2$ further vertices in $D$.
\item For every vertex $u$ of degree $1$ in $H$, 
the set $D$ contains the four vertices of the corresponding subgraph 
encircled in Figure \ref{figd1}.
This yields $4n_1$ further vertices in $D$.
\end{enumerate}
The construction implies that $D$ is a dissociation set in $G$, and
$$|D|
=2k+2(n_1+n_2+n_3-2k)+2n_2+4n_1
=6n_1+4n_2+2n_3-2k
=10k.$$
By Theorem \ref{theorem1},
we have 
$6k\geq \alpha(G)\geq \frac{3}{5}{\rm diss}(G)\geq \frac{3}{5}|D|=6k$,
which implies $\alpha(G)=6k$ and ${\rm diss}(G)=10k$.
\end{proof}
We proceed to our final main result, 
the characterization of the extremal graphs for Theorem \ref{theorem1}.

\begin{theorem}\label{conjecture1}
A connected cubic graph $G$ satisfies $\alpha(G)=\frac{3}{5}{\rm diss}(G)$
if and only if $G\in \bigcup\limits_{k\in\mathbb{N}}{\cal G}_k$.
\end{theorem}
\begin{proof}
Let $G$ be a connected cubic graph of order $n$ at least $6$ 
with $\alpha(G)=\frac{3}{5}{\rm diss}(G)$.
In view of Lemma \ref{lemma1}, it suffices to show that $G$ belongs to 
$\bigcup\limits_{k\in\mathbb{N}}{\cal G}_k$.
Let $D$, $\bar{D}$, $p$, $q$, $r$, and $s$ be exactly as in the proof of Theorem \ref{theorem1}. Recall that $D$ was chosen to maximize $p$,
and that, using that $G$ is extremal, we obtained that $p=0$.
This implies that every maximum dissociation set of $G$ is a suitable choice for $D$,
and, hence, every maximum dissociation set shares the following properties with $D$,
where $k=\frac{n}{18}$: 
\begin{eqnarray}\label{eext0}
\mbox{$p=0$, $q=5k$, $r=4k$, $s=2k$, and $\alpha(G)=r+s$.}
\end{eqnarray}
Let the set $\bar{D}_0$ contain the $r=2k$ isolated vertices of $G[\bar{D}]$,
and let the set $\bar{D}_1$ contain the remaining $2s=4k$ vertices from $\bar{D}$.
A path or cycle in $G$ is {\it alternating} 
if it alternates between the vertices from $D$ and $\bar{D}$.

\begin{claim}\label{claim1}
No induced alternating path has both its endpoints in $\bar{D}_1$.
\end{claim}
\begin{proof}[Proof of Claim \ref{claim1}.]
Suppose, for a contradiction, that 
$P:\bar{u}_1u_1\bar{u}_2u_2\ldots \bar{u}_\ell u_\ell\bar{u}_{\ell+1}$
is an induced alternating path with 
$\bar{u}_1,\bar{u}_{\ell+1}\in \bar{D}_1$.
Since $P$ is induced, the vertices 
$\bar{u}_1$ and $\bar{u}_{\ell+1}$ 
are not adjacent.
Let $\bar{v}_1$ and $\bar{v}_{\ell+1}$ be the neighbors of 
$\bar{u}_1$ and $\bar{u}_{\ell+1}$ in $\bar{D}_1$,
respectively.
Since $P$ is induced, the set $I$ containing
$\bar{v}_1,u_1,\ldots,u_\ell, \bar{v}_{\ell+1}$
as well as one vertex from every component of $G[\bar{D}]$ 
that does not intersect $P$
is independent.
Since $P$ intersects at most $\ell+1$ of the $r+s$ components of $G[\bar{D}]$,
we obtain that $|I|=1+\ell+1+(r+s)-(\ell+1)=r+s+1
\stackrel{(\ref{eext0})}{>}\alpha(G)$, 
which is a contradiction.
\end{proof}

\begin{claim}\label{claim2}
No induced alternating cycle contains a vertex from $\bar{D}_1$.
\end{claim}
\begin{proof}[Proof of Claim \ref{claim2}.]
Suppose, for a contradiction, that 
$C:\bar{u}_1u_1\bar{u}_2u_2\ldots \bar{u}_\ell u_\ell\bar{u}_1$
is an induced alternating cycle with 
$\bar{u}_1\in \bar{D}_1$.
Let $\bar{v}_1$ be the neighbor of $\bar{u}_1$ in $\bar{D}_1$.
Since $C$ is induced, the vertex $\bar{v}_1$ does not lie on $C$,
and the set $I$ containing
$\bar{v}_1,u_1,\ldots,u_\ell$
as well as one vertex from every component of $G[\bar{D}]$ 
that does not intersect $C$
is independent.
Since $C$ intersects at most $\ell$ of the $r+s$ components of $G[\bar{D}]$,
we obtain that $|I|=1+\ell+(r+s)-\ell=r+s+1
\stackrel{(\ref{eext0})}{>}\alpha(G)$, 
which is a contradiction.
\end{proof}
For a set $X$ of vertices of $G$, let $d(X)=|X\cap D|$ and 
$\bar{d}(X)=|X\cap \bar{D}|$.
For a subgraph $H$ of $G$, let $d(H)=d(V(H))$ 
and $\bar{d}(H)=\bar{d}(V(H))$.
A {\it $\bar{D}_0$-alternating cycle} is a (not necessarily induced) alternating cycle in $G$
that does not contain a vertex from $\bar{D}_1$.
For a $\bar{D}_0$-alternating cycle $C$, 
the set $N_G\left[V(C)\cap \bar{D}\right]$
is the corresponding {\it alternating cycle set}.
If $X=N_G\left[V(C)\cap \bar{D}\right]$ is an alternating cycle set, 
then every vertex in $V(C)\cap \bar{D}$ 
has at least two of its three neighbors in $V(C)$
and at most one neighbor outside of $V(C)$, which implies
\begin{eqnarray}\label{eext1}
d(X) & \leq & 2\bar{d}(X).
\end{eqnarray}
We now apply a recursive reduction to $G$:
As long as the current graph contains a $\bar{D}_0$-alternating cycle,
iteratively remove the corresponding alternating cycle set.
Let $G'$ be the resulting induced subgraph of $G$,
that is, the graph $G'$ contains no $\bar{D}_0$-alternating cycle.
By construction, 
all vertices from $\bar{D}_1$ belong to $G'$, and,
for every vertex from $D$ in $G'$, 
its two neighbors in $\bar{D}$ belong to $G'$.

Let $uv$ be one of the $s=2k$ edges in $G[\bar{D}]$.
A {\it special $uv$-cycle} is an induced cycle $C$ in $G'$ that contains $uv$ 
such that $C-uv$ is an alternating path.
Let the {\it bubble $B_{uv}$ of $uv$} 
be the union of $uv$ and all special $uv$-cycles.
Claim \ref{claim1} implies
$(V(B_{uv})\cap \bar{D})\setminus \{ u,v\} \subseteq \bar{D}_0$.
Our goal is to show that the bubbles of the edges in $G[\bar{D}]$ 
are exactly the four different graphs of orders $3$, $6$, and $9$
shown in Figures \ref{figd3}, \ref{figd2}, and \ref{figd1}
replacing certain vertices,
in particular, the bubbles are disjoint.

By definition, 
\begin{itemize}
\item every vertex in $V(B_{uv})\cap D$ sends exactly two edges to $V(B_{uv})\cap \bar{D}$,
\item each of the two vertices $u$ and $v$ sends at most two edges to $V(B_{uv})\cap D$,
and 
\item every vertex in $(V(B_{uv})\cap \bar{D})\setminus \{ u,v\}$ 
sends at most three edges to $V(B_{uv})\cap D$.
\end{itemize}
Double-counting the edges between $V(B_{uv})\cap D$ and $V(B_{uv})\cap \bar{D}$,
these observations imply
\begin{eqnarray}\label{eext2}
2d(B_{uv}) & \leq & 3\bar{d}(B_{uv})-2.
\end{eqnarray}

\begin{claim}\label{claim3}
$d(B_{uv}) \leq 2\bar{d}(B_{uv})-3$ for every edge $uv$ in $G[\bar{D}]$.
\end{claim}
\begin{proof}[Proof of Claim \ref{claim3}.]
By (\ref{eext2}), we have 
$\bar{d}(B_{uv})\geq \left\lceil \frac{2}{3}(d(B_{uv})+1)\right\rceil$.
For $d(B_{uv})\geq 3$, 
this implies $\bar{d}(B_{uv})\geq \frac{1}{2}(d(B_{uv})+3)$,
and, hence, the claim.
Since $\bar{d}(B_{uv})\geq |\{ u,v\}|=2$,
the claim holds for $d(B_{uv})\leq 1$.
Now, let $d(B_{uv})=2$.
Clearly, we may assume $\bar{d}(B_{uv})=2$.
Let $V(B_{uv})\cap D=\{ u',v'\}$.
By the definition of $B_{uv}$,
the special $uv$-cycles are two triangles, and, hence,
there are all possible four edges between $\{ u,v\}$ and $\{ u',v'\}$.
Since $G$ is not a $K_4$, the vertices $u'$ and $v'$ are not adjacent.
Now, the set $D'=(D\setminus \{ u',v'\})\cup \{ u,v\}$
is a maximum dissociation set of $G$ 
such that $G[D']$ has two isolated vertices, 
which contradicts (\ref{eext0}).
\end{proof}
Let $uv$ be an edge in $G[\bar{D}]$.

An induced subgraph $H$ of $G'$ is a {\it bubble extension} of $B_{uv}$ if
\begin{itemize}
\item there is a special $uv$-cycle $C$
and 
\item an induced alternating path $P$
between a vertex $x$ from $V(C)\cap \bar{D}_0$
and a vertex $y$ from $D$ such that 
\begin{itemize}
\item $V(P)\not\subseteq V(B_{uv})$,
\item $x$ is the only common vertex of $C$ and $P$,
\item $y$ is adjacent to some vertex $z$ from $V(C)\cap D$, 
\item $H=G[V(C)\cup V(P)]$, and 
\item $E(H)=E(C)\cup E(P)\cup \{ yz\}$.
\end{itemize}
\end{itemize}
Note 
that $x$ and $z$ are the only vertices of $H$ that are of degree $3$ in $H$,
that all remaining vertices of $H$ have degree $2$ in $H$, and 
that $d(H)=\bar{d}(H)$.

\begin{claim}\label{claim4}
No bubble has a bubble extension.
\end{claim}
\begin{proof}[Proof of Claim \ref{claim4}.]
Suppose, for a contradiction, that $H$ is a bubble extension of the bubble $B_{uv}$.
Let $C$, $P$, $x$, $y$, and $z$ be as above.
By symmetry, we may assume that $z$ lies on the unique path in $C$ 
between $x$ and $v$ that avoids $u$.
This implies that, for every vertex $w$ in $(V(H)\cap \bar{D})\setminus \{ u,v\}$,
there is an induced alternating path between $w$ and $u$,
and Claim \ref{claim1} implies
$(V(H)\cap \bar{D})\setminus \{ u,v\} \subseteq \bar{D}_0$.

Let $H'$ be an inclusionwise maximal induced subgraph of $G'$ such that 
the following conditions (i), (ii), and (iii) hold:
\begin{enumerate}[(i)]
\item 
\begin{itemize}
\item $H\subseteq H'$,
\item $u$ and $v$ have degree $2$ in $H'$,
\item $\left(V\left(H'\right)\cap \bar{D}\right)\setminus \{ u,v\} 
\subseteq \bar{D}_0$,
\item $yz$ is the only edge in $G'\left[V\left(H'\right)\cap D\right]$, and
\item $d(H')=\bar{d}(H')$.
\end{itemize}
\item For every vertex $w$ in $V\left(H'\right)\setminus \{ u,v\}$,
the graph $H'$ contains 
an alternating path $P_u(w)$ between $w$ and $u$ as well as 
an alternating path $P_v(w)$ between $w$ and $v$ such that 
\begin{itemize}
\item $P_u(w)$ is induced and
\item either $P_v(w)$ is induced 

or 
$E[G[V(P_v(w))]]\setminus E(P_v(w))=\{ yz\}$, 
and $P$ is a subpath of $P_u(w)$ as well as of $P_v(w)$.
\end{itemize}
\item Exactly one vertex $y'$ 
from $V\left(H'\right)\cap D$
has exactly one neighbor in $V\left(H'\right)\cap \bar{D}$,
and all remaining vertices from $V\left(H'\right)\cap D$
have exactly two neighbors in $V\left(H'\right)\cap \bar{D}$.
\end{enumerate}
Since $(H,y)$ satisfies all properties required for $\left(H',y'\right)$ 
in (i), (ii), and (iii),
the graph $H'$ is well defined.
The proof is completed by showing that $H'$ is a proper subgraph
of an induced subgraph $H''$ of $G'$ satisfying (i), (ii), and (iii),
contradicting the maximality of $H'$.

Let $D'=V\left(H'\right)\cap D$ and
$\bar{D}'=V\left(H'\right)\cap \bar{D}$.

By (iii), the vertex $y'$ has exactly one neighbor $x'$ in $\bar{D}\setminus \bar{D}'$.
Note that $y'$ is the only neighbor of $x'$ in $V(H')$, and, hence,
extending the two paths $P_u(y')$ and $P_v(y')$ guaranteed for $y'$ in (ii) 
by the edge towards $x'$, yields two paths $P_u(x')$ and $P_v(x')$ 
with the properties stated in (ii) also for $x'$.
By Claim \ref{claim1}, the induced alternating path $P_u(x')$ between $x'$ and $u$
implies that $x'$ belongs to $D_0$.

By (i), (iii), and the choice of $x'$, 
there are exactly $2d(H')-2$ edges in $G'$ between
$D'$ and $\{ x'\}\cup (\bar{D}'\setminus \{ u,v\})$.
Since, by (i), every vertex in $\{ x'\}\cup (\bar{D}'\setminus \{ u,v\})$ 
sends exactly three edges to $D$,
there are exactly
$$3|\{ x'\}\cup (\bar{D}'\setminus \{ u,v\})|-2d\left(H'\right)+2=
3\left(\bar{d}\left(H'\right)-1\right)-2d\left(H'\right)+2
\stackrel{(i)}{=}d\left(H'\right)-1$$ edges between
$\{ x'\}\cup (\bar{D}'\setminus \{ u,v\})$
and $D\setminus D'$.

If some vertex in $D\setminus D'$ has two neighbors 
in $\{ x'\}\cup (\bar{D}'\setminus \{ u,v\})$,
say $a'$ and $a''$, 
then the two paths $P_u(a')$ and $P_u(a'')$ 
share at least the edge incident with $u$,
and, hence, their union contains a $\bar{D}_0$-alternating cycle,
contradicting the choice of $G'$.
Hence, every vertex in $D\setminus D'$ has at most 
one neighbor in $\{ x'\}\cup (\bar{D}'\setminus \{ u,v\})$.

Since $yz$ is an edge in $D'$,
and $|D'\setminus \{ y,z\}|=d\left(H'\right)-2$,
this implies the existence of some vertex $y''$ in $D\setminus D'$
that has exactly one neighbor in $\{ x'\}\cup (\bar{D}'\setminus \{ u,v\})$,
say $x''$, and no neighbor in $D'$.

If $y''$ is not adjacent to $u$ or $v$,
then $H''=G[V(H')\cup \{ x',y''\}]$ satisfies (i), (ii), and (iii),
contradicting the choice of $H'$.
Note that extending the two paths $P_u(x'')$ and $P_v(x'')$
guaranteed for $x''$ in (ii) by the edge towards $y''$, 
yields two paths $P_u(y'')$ and $P_v(y'')$ 
with the properties stated in (ii) also for $y''$.
Hence, we may assume that $y''$ is adjacent to $u$ or $v$.

If $y''$ is adjacent to $u$,
then $P_u(x'')$ together with $y''$ and the two edges 
$uy''$ and $x''y''$ yields an induced alternating cycle 
that contains a vertex from $\bar{D}_1$,
contradicting Claim \ref{claim2}.
If $y''$ is adjacent to $v$, and $P_v(x'')$ is induced,
then $P_v(x'')$ together with $y''$ and the two edges 
$vy''$ and $x''y''$ yields an induced alternating cycle 
that contains a vertex from $\bar{D}_1$,
contradicting Claim \ref{claim2}.
Hence, we may assume that $y''$ is adjacent to $v$,
and that $P_v(x'')$ is not induced.
By (ii), this implies that $P_u(x'')$ contains $P$ as a subpath.
Now, the path $P_u(x'')$ together with the vertices $v$ and $y''$ 
and the edges $uv$, $vy''$, and $x''y''$ 
yields a special $uv$-cycle.
By the definition of the bubble $B_{uv}$,
this implies that all vertices of $P$ belong to $B_{uv}$,
contradicting the condition $V(P)\not\subseteq V(B_{uv})$
in the definition of the bubble extension $H$.
This completes the proof.
\end{proof}

\begin{claim}\label{claim5}
Every two bubbles are disjoint.
\end{claim}
\begin{proof}[Proof of Claim \ref{claim5}.]
Suppose, for a contradiction, 
that the two bubbles $B_{uv}$ and $B_{u'v'}$ are not disjoint,
where $uv$ and $u'v'$ are distinct edges in $G[\bar{D}]$.
By Claim \ref{claim1}, we obtain 
$u',v'\not\in V(B_{uv})$ and
$u,v\not\in V(B_{u'v'})$.
Let $x\in V(B_{uv})\cap V(B_{u'v'})$.
Let $C$ be a special $uv$-cycle and let $C'$ be a special $u'v'$-cycle 
such that $x$ lies on $C$ and on $C'$.
Let $P$ be a shortest path in $C'$ between $\{ u',v'\}$ and $V(C)$.
By Claim \ref{claim1}, 
there is at least one edge between $V(P)\cap D$ and $V(C)\cap D$.
This implies that $C\cup P$ contains a bubble extension of $B_{uv}$,
contradicting Claim \ref{claim4}.
\end{proof}
Let 
$$G''=G'-\bigcup_{uv\in E(G[\bar{D}])}V(B_{uv}).$$
\begin{claim}\label{claim6}
$\bar{D}\cap V(G'')$ is a maximum independent set in $G''$,
and $d(G'')\leq 2\bar{d}(G'')$.
\end{claim}
\begin{proof}[Proof of Claim \ref{claim6}.]
Since $\alpha(G'')\geq \frac{1}{2}d(G'')$, 
the first part of the statement implies the second.
Therefore, suppose, for a contradiction, 
that the independent set $\bar{D}\cap V(G'')$ 
is not a maximum independent set in $G''$.
This implies that a maximum independent set $I''$ in $G''$ intersects $D$.
Let $I''$ be chosen such that $|I''\cap D|$ is as small as possible.
Let $F=G\left[I''\cup \left(V(G'')\cap \bar{D}\right)\right]$.
Since $G''$ contains no $\bar{D}_0$-alternating cycle,
the graph $F$ is a forest.

If $V(G'')\cap \bar{D}$ is not empty,
then let $T$ be a component of $F$ that contains a vertex from $\bar{D}$.
By the choice of $I''$, we have $|V(T)\cap D|>|V(T)\cap \bar{D}|$,
which implies that $T$ contains two distinct leaves from $D$.
In this case, let $P$ be a path in $T$ 
between two distinct leaves $x$ and $x'$ of $T$ that belong to $D$.
If $V(G'')\cap \bar{D}$ is empty, then
let $P$ be a path of length $0$ consisting of a vertex $x=x'$ from $I''$.

There are two distinct edges $xy$ and $x'y'$ with $y,y'\in \bar{D}\setminus V(G'')$.
By the definition of $G'$, it follows that $y$ and $y'$ both lie in bubbles,
say $y\in B_{uv}$ and $y'\in B_{u'v'}$, where $uv=u'v'$ is possible.

If $uv\not=u'v'$, 
then $B_{uv}\cup B_{u'v'}\cup P$ contains 
an induced alternating path with both endpoints in $\bar{D}_1$
or a bubble extension of $B_{uv}$ or $B_{u'v'}$,
which implies a contradiction to Claim \ref{claim1} or Claim \ref{claim4}.
Hence, we obtain that $uv=u'v'$.
By Claim \ref{claim2} and the definition of $B_{uv}$,
we may assume, by symmetry, that $y$ is distinct from $u$ and $v$.
Let $C$ be a special $uv$-cycle containing $y$.

First, suppose that $y'=u$.
By Claim \ref{claim2}, the alternating cycle contained in $C\cup P$
that contains $P$ and intersects $\bar{D}_1$ in $u$ is not induced.
By Claim \ref{claim4},
this implies that $x'$ has its neighbor in $D$ 
on the path in $C$ between $y$ and $u$ that avoids $v$.
Since the path in $C$ between $y$ and $v$ that avoids $u$
together with $P$ and the edges $uv$, $xy$, and $x'y'$ 
is not a special $uv$-cycle,
the union of $C$ and $P$ contains a bubble extension,
contradicting Claim \ref{claim4}.
Hence, we may assume, by symmetry, that $y'$ is distinct from $u$ and $v$.
Let $C'$ be a special $u'v'$-cycle containing $y'$.
Since $G'$ contains no $\bar{D}_0$-alternating cycle, 
we obtain $V(C)\cap V(C')=\{ u,v\}$,
that is, in particular, $y'\not\in V(C)$ and $y\not\in V(C')$.
By Claim \ref{claim4}, 
there is no edge 
between $(V(C)\cup V(C'))\cap D$ and $V(P)\cap D$.
Let $P_u$ be the path in $C$ between $y$ and $u$ avoiding $v$,
let $P_v$ be the path in $C$ between $y$ and $v$ avoiding $u$,
let $P'_u$ be the path in $C'$ between $y'$ and $u$ avoiding $v$, and
let $P'_v$ be the path in $C'$ between $y'$ and $v$ avoiding $u$.
By Claim \ref{claim2}, 
there is an edge 
between $V(P_u)\cap D$ and $V(P'_u)\cap D$
as well as 
between $V(P_v)\cap D$ and $V(P'_v)\cap D$.
By the definition of $B_{uv}$ and since no vertex of $P$ belongs to $B_{uv}$,
there is an edge 
between $V(P_u)\cap D$ and $V(P'_v)\cap D$
as well as 
between $V(P_v)\cap D$ and $V(P'_u)\cap D$.
Now, there is a bubble extension of $B_{uv}$
using $C$, $P$ and a proper subpath of $P_u'$,
contradicting Claim \ref{claim4}.
\end{proof}
Let ${\cal X}$ be the collection of all alternating cycle sets 
that were recursively removed from $G$ to construct $G'$.
By (\ref{eext0}), (\ref{eext1}), Claim \ref{claim3}, and Claim \ref{claim6}, we obtain
\begin{eqnarray*}
10k &=& d(G)\\
&=& \sum_{X\in {\cal X}}d(X)+\sum_{e\in E(G[\bar{D}_1])}d(B_e)+d(G'')\\
&\leq & \sum_{X\in {\cal X}}2\bar{d}(X)
+\sum_{e\in E(G[\bar{D}_1])}(2\bar{d}(B_e)-3)
+2\bar{d}(G'')\\
&=& 2\bar{d}(G)-3s\\
&=& 2(r+2s)-3s\\
&=& 10k.
\end{eqnarray*}
Since equality holds throughout this inequality chain, 
each individual estimate must be satisfied with equality, that is,
\begin{itemize}
\item $d(X)=2\bar{d}(X)$ for every alternating cycle set $X$ in ${\cal X}$,
\item $d(B_e)=2\bar{d}(B_e)-3$ for every edge $e$ of $G[\bar{D}_1]$, and
\item $d(G'')=2\bar{d}(G'')$.
\end{itemize}
Since, by Claim \ref{claim6},
we have $d(G'')\leq {\rm diss}(G'')\leq 2\alpha(G'')=2\bar{d}(G'')=d(G'')$,
the subcubic graph $G''$ satisfies ${\rm diss}(G'')=2\alpha(G'')$,
and, therefore, each of its components belongs to ${\cal G}$ 
as described in Theorem \ref{theorem3}.
If $G''$ contains a copy of $K^*_4$, 
then this subgraph contains no two adjacent vertices from $\bar{D}_1$,
which easily implies the existence of a $\bar{D}_0$-alternating cycle in that subgraph,
contradicting the construction of $G'$.
Since $d(G'')=2\bar{d}(G'')$, 
the graph $G''$ arises from the union of disjoint copies of $K_3$
by adding additional edges as described for Theorem \ref{theorem3}.
Since $G''$ does not contain $\bar{D}_0$-alternating cycles,
these additional edges are actually all bridges in $G''$,
that is, the disjoint copies of $K_3$ are the only cycles in $G''$.
Note that, while copies of $K_3$ in a graph as in ${\cal G}$
contribute twice as many vertices to $D$ than to $\bar{D}$,
copies of $K_2$ only contribute vertices to $D$,
which implies that the construction of $G''$ 
does not use any copy of $K_2$.

\begin{claim}\label{claim7}
All bubbles are isomorphic to 
$K_3$,
$K_4^*$, or one of the two graphs in Figure \ref{figb}.
\begin{figure}[H]
\begin{center}
\unitlength .6mm 
\linethickness{0.4pt}
\ifx\plotpoint\undefined\newsavebox{\plotpoint}\fi 
\begin{picture}(107,38)(0,0)
\put(8,0){\circle*{2}}
\put(78,0){\circle*{2}}
\put(28,0){\circle*{2}}
\put(98,0){\circle*{2}}
\put(8,0){\line(1,0){20}}
\put(78,0){\line(1,0){20}}
\put(0,10){\circle*{2}}
\put(70,10){\circle*{2}}
\put(36,10){\circle*{2}}
\put(106,10){\circle*{2}}
\put(0,22){\circle*{2}}
\put(70,22){\circle*{2}}
\put(36,22){\circle*{2}}
\put(106,22){\circle*{2}}
\put(8,32){\circle*{2}}
\put(78,32){\circle*{2}}
\put(28,32){\circle*{2}}
\put(98,32){\circle*{2}}
\put(18,37){\circle*{2}}
\put(88,37){\circle*{2}}
\put(18,37){\line(-2,-1){10}}
\put(88,37){\line(-2,-1){10}}
\put(8,32){\line(-4,-5){8}}
\put(78,32){\line(-4,-5){8}}
\put(0,22){\line(0,-1){12}}
\put(70,22){\line(0,-1){12}}
\put(0,10){\line(4,-5){8}}
\put(70,10){\line(4,-5){8}}
\put(18,37){\line(2,-1){10}}
\put(88,37){\line(2,-1){10}}
\put(28,32){\line(4,-5){8}}
\put(98,32){\line(4,-5){8}}
\put(36,22){\line(0,-1){12}}
\put(106,22){\line(0,-1){12}}
\put(36,10){\line(-4,-5){8}}
\put(106,10){\line(-4,-5){8}}
\multiput(0,22)(.055944056,-.153846154){143}{\line(0,-1){.153846154}}
\multiput(70,22)(.055944056,-.153846154){143}{\line(0,-1){.153846154}}
\multiput(36,22)(-.055944056,-.153846154){143}{\line(0,-1){.153846154}}
\multiput(106,22)(-.055944056,-.153846154){143}{\line(0,-1){.153846154}}
\multiput(36,10)(-.055944056,.153846154){143}{\line(0,1){.153846154}}
\multiput(8,32)(-.055944056,-.153846154){143}{\line(0,-1){.153846154}}
\multiput(78,32)(.0714285714,-.056122449){392}{\line(1,0){.0714285714}}
\multiput(98,32)(-.0714285714,-.056122449){392}{\line(-1,0){.0714285714}}
\end{picture}
\end{center}
\caption{Two bubbles of order $9$.}\label{figb}
\end{figure}
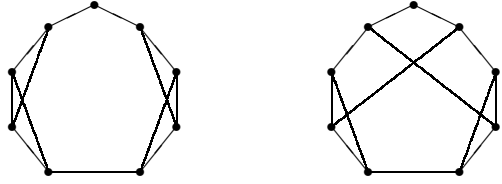
\end{claim}
\begin{proof}[Proof of Claim \ref{claim7}.]
Let $uv$ be an edge of $G[\bar{D}_1]$.
Since
$d(B_{uv})=2\bar{d}(B_{uv})-3\stackrel{(\ref{eext2})}{\geq}\frac{4}{3}(d(B_{uv})+1)-3$, we obtain that $d(B_{uv})\in \{ 1,3,5\}$.

If $d(B_{uv})=1$, then $\bar{d}(B_{uv})=2$, and $B_{uv}$ is a triangle.

Now, let $d(B_{uv})=3$, and, hence, $\bar{d}(B_{uv})=3$.
If $V(B_{uv})\cap D$ is independent, then, 
since the vertices in $V(B_{uv})\cap D$ 
have all their neighbors in $\bar{D}$ within the bubble $B_{uv}$, 
the graph $G$ has an independent set containing 
the three vertices from $V(B_{uv})\cap D$
as well as
$r+s-2$ further vertices, one from each component of $G[\bar{D}]$
that does not intersect $B_{uv}$,
which yields a contradiction to (\ref{eext0}).
Hence, the set $V(B_{uv})\cap D$ contains exactly one edge.
Let $V(B_{uv})\cap D=\{ u_1,u_2,u_3\}$ such that $u_1u_2$ is an edge,
and let $V(B_{uv})\cap \bar{D}=\{ u,v,w\}$.
Since $w$ lies on a special $uv$-cycle, 
we obtain, by symmetry, 
that $w$ is adjacent to $u_2$ and $u_3$,
that $u_2$ is adjacent to $u$, and
that $u_3$ is adjacent to $v$.
Since $u_1$ lies on a special $uv$-cycle, 
the vertex $u_1$ is adjacent to 
either $u$ and $v$ (see the right part of Figure \ref{figb2}),
or $u$ and $w$ (see the left part of Figure \ref{figb2}), 
or $v$ and $w$.
In the third case, the independent set $\{ u,u_1,u_3\}$ 
can be extended to an independent set in $G$ with more than $r+s$ vertices,
contradicting (\ref{eext0}).
Altogether, the bubble $B_{uv}$ is isomorphic to $K_4^*$.

Finally, let $d(B_{uv})=5$, and, hence, $\bar{d}(B_{uv})=4$.
Similarly as above, we obtain that $V(B_{uv})\cap D$ contains exactly two edges.
In view of the definition of a bubble, 
it is again easy to verify that $B_{uv}$ 
is isomorphic to one of the two graphs in Figure \ref{figb}.
Counting the edges of $B_{uv}$ between $D$ and $\bar{D}$
actually implies that there is only one edge of $G$ leaving $V(B_{uv})$, 
and that both endpoints of this edge belong to $D$.
\end{proof}
By the final remark of the previous proof, 
the unique vertices of degree two in the bubbles $B_e$
of order $9$ shown in Figure \ref{figb} both lie in $D$, 
and their neighbors outside of $B_e$ also lie in $D$.
For bubbles $B_e$ that are copies of $K_4^*$, 
these observations easily imply that $V(B_e)\cap D$ 
is as shown in Figure \ref{figb2}.

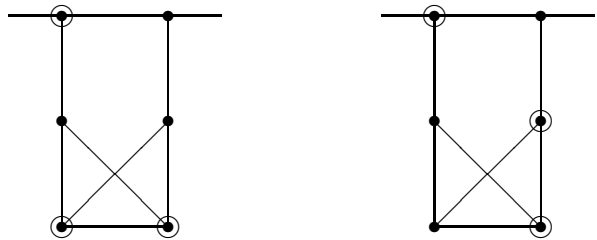
\begin{figure}[H]
\begin{center}
\unitlength 0.7mm 
\linethickness{0.4pt}
\ifx\plotpoint\undefined\newsavebox{\plotpoint}\fi 
\begin{picture}(110,42)(0,0)
\put(10,0){\circle*{2}}
\put(80,0){\circle*{2}}
\put(10,20){\circle*{2}}
\put(80,20){\circle*{2}}
\put(10,40){\circle*{2}}
\put(80,40){\circle*{2}}
\put(30,0){\circle*{2}}
\put(100,0){\circle*{2}}
\put(30,20){\circle*{2}}
\put(100,20){\circle*{2}}
\put(30,40){\circle*{2}}
\put(100,40){\circle*{2}}
\put(10,0){\line(1,0){20}}
\put(80,0){\line(1,0){20}}
\put(10,40){\line(1,0){20}}
\put(80,40){\line(1,0){20}}
\put(10,40){\line(0,-1){40}}
\put(80,40){\line(0,-1){40}}
\put(10,0){\line(1,1){20}}
\put(80,0){\line(1,1){20}}
\put(30,20){\line(0,-1){20}}
\put(100,20){\line(0,-1){20}}
\put(30,0){\line(-1,1){20}}
\put(100,0){\line(-1,1){20}}
\put(30,40){\line(0,-1){20}}
\put(100,40){\line(0,-1){20}}
\put(30,40){\line(1,0){10}}
\put(100,40){\line(1,0){10}}
\put(10,40){\line(-1,0){10}}
\put(80,40){\line(-1,0){10}}
\put(10,0){\circle{4}}
\put(100,20){\circle{4}}
\put(30,0){\circle{4}}
\put(100,0){\circle{4}}
\put(10,40){\circle{4}}
\put(80,40){\circle{4}}
\end{picture}
\end{center}
\caption{Possible intersections up to symmetry of bubbles $B_e$ 
that are isomorphic to $K_4^*$;
the vertices that belong to $D$ are encircled.}\label{figb2}
\end{figure}

\begin{claim}\label{claim8}
For $X\in {\cal X}$ with $|X|>6$,
the set $\bar{D}\cap X$ is a maximum independent set in $G[X]$.
\end{claim}
\begin{proof}[Proof of Claim \ref{claim8}]
By definition and since $d(X)=2\bar{d}(X)$,
the induced subgraph $G[X]$ of $G$ contains a $\bar{D}_0$-alternating cycle 
$C:u_1\bar{u}_1u_2\bar{u}_2\dots u_{\ell}\bar{u}_{\ell}u_1$,
where $\{\bar{u}_1,\ldots,\bar{u}_{\ell}\}=\bar{D}\cap X$,
and the set $W$ defined as $X\setminus V(C)$
consists of $\ell$ distinct vertices $w_1,\ldots,w_{\ell}$ from $D$,
where $w_i$ is the third neighbor of $u_i$ in $D$ for $i\in [\ell]$.
Since $|X|>6$, we have $\ell\geq 3$.

First, we show that, for every $i\in [\ell]$,
there is an induced alternating path $P_i$ between $w_i$ 
and a vertex from $\bar{D}$ in some bubble $B_i$,
where $V(P_i)\cap X=\{ w_i\}$ and $|V(P_i)\cap V(B_i)|=1$.
Let $i\in [\ell]$.
Let $y_i$ be the neighbor of $w_i$ in $\bar{D}$ that is distinct from $\bar{u}_i$.
Since $w_i\in X$, the vertex $y_i$ 
cannot belong to some alternating cycle set in ${\cal X}$.
If $y_i$ belongs to some bubble, then $P_i:w_iy_i$ is the desired path.
Otherwise, the vertex $y_i$ belongs to $G''$, and, hence, 
lies in a triangle $y_iw_i'w_i''$ in $G''$ with $w_i',w_i''\in D$.
Let $y_i'$ be the neighbor of $w'_i$ in $\bar{D}$ that is distinct from $y_i$.
Again, either $y_i'$ belongs to some bubble, 
in which case $P_i:w_iy_iw_i'w_i$ is the desired path,
or $y_i'$ lies in a second triangle in $G''$.
Continuing this reasoning, it follows that the desired path $P_i$ 
can be obtained using edges in distinct triangles in $G''$ 
as illustrated in Figure \ref{figpi}.

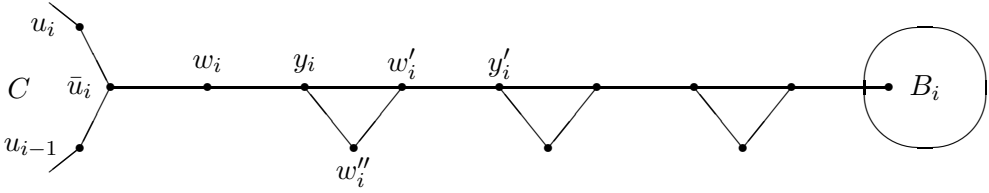
\begin{figure}[H]
\begin{center}
\unitlength 0.8mm 
\linethickness{0.4pt}
\ifx\plotpoint\undefined\newsavebox{\plotpoint}\fi 
\begin{picture}(164,29)(0,0)
\put(20,15){\circle*{1.5}}
\put(36,15){\circle*{1.5}}
\put(52,15){\circle*{1.5}}
\put(60,5){\circle*{1.5}}
\put(68,15){\circle*{1.5}}
\put(84,15){\circle*{1.5}}
\put(92,5){\circle*{1.5}}
\put(100,15){\circle*{1.5}}
\put(116,15){\circle*{1.5}}
\put(124,5){\circle*{1.5}}
\put(132,15){\circle*{1.5}}
\put(148,15){\circle*{1.5}}
\put(15,5){\circle*{1.5}}
\put(15,25){\circle*{1.5}}
\put(15,25){\line(1,-2){5}}
\put(20,15){\line(-1,-2){5}}
\put(15,5){\line(-5,-4){5}}
\put(15,25){\line(-5,4){5}}
\put(52,15){\line(1,0){16}}
\put(84,15){\line(1,0){16}}
\put(116,15){\line(1,0){16}}
\put(68,15){\line(-4,-5){8}}
\put(100,15){\line(-4,-5){8}}
\put(132,15){\line(-4,-5){8}}
\put(60,5){\line(-4,5){8}}
\put(92,5){\line(-4,5){8}}
\put(124,5){\line(-4,5){8}}
\put(52,15){\line(-1,0){16}}
\put(36,15){\line(-1,0){16}}
\put(68,15){\line(1,0){80}}
\put(36,19){\makebox(0,0)[cc]{$w_i$}}
\put(52,19){\makebox(0,0)[cc]{$y_i$}}
\put(68,19){\makebox(0,0)[cc]{$w_i'$}}
\put(60,1){\makebox(0,0)[cc]{$w_i''$}}
\put(84,19){\makebox(0,0)[cc]{$y_i'$}}
\put(15,15){\makebox(0,0)[cc]{$\bar{u}_i$}}
\put(7,5){\makebox(0,0)[cc]{$u_{i-1}$}}
\put(9,25){\makebox(0,0)[cc]{$u_i$}}
\put(154,15){\oval(20,20)[]}
\put(154,15){\makebox(0,0)[cc]{$B_i$}}
\put(5,15){\makebox(0,0)[cc]{$C$}}
\end{picture}
\end{center}
\caption{An induced alternating path $P_i$ between $w_i$ and some vertex in a bubble $B_i$.}\label{figpi}
\end{figure}
Note that, if $y_i$ does not belong to a bubble but to the triangle $y_iw_i'w_i''$
as in Figure \ref{figpi},
then we have a choice of using either $w_i'$ or $w_i''$ to construct the path $P_i$.
This leads to an alternative choice $P_i'$ for the path $P_i$
that is disjoint from $P_i$ after $y_i$, and leads to some bubble $B_i'$.
Note that $B_i$ can only coincide with $B_i'$ if $B_i$ is a triangle,
and the two paths $P_i$ and $P_i'$ use the two disjoint edges into $B_i$
that are incident with the vertices of $B_i$ from $\bar{D}$.

If the unique vertex in $V(P_i)\cap V(B_i)$ belongs to $\bar{D}_1$,
then let $Q_i=P_i$.
Otherwise, the bubble $B_i$ and its intersection with $D$
is necessarily as in the right part of Figure \ref{figb2}.
In this case, let the path $Q_i$ arise by extending $P_i$ 
by one vertex in $V(B_i)\cap D$ that has no neighbors outside of $V(B_i)$,
and one vertex in $V(B_i)\cap \bar{D}_1$.

\bigskip

\noindent In view of the claimed statement, we suppose, for a contradiction,
that $G[X]$ has a maximum independent set $I$ that is 
strictly bigger than $\bar{D}\cap X$.
Let $I$ be chosen such that $|I\cap D|$ is as small as possible.
Let $F=G[I\cup (\bar{D}\cap X)]$.

First, suppose that $F$ is not a forest.
By construction, the only possible cycle in $F$ is $C$,
that is, $F$ consists of $C$ and at least one vertex, say $w_1$, from $W$.
Now, the union of $I'=(V(Q_1)\cup V(C))\cap D$ 
with a maximal independent set in $G\left[\bar{D}\setminus N_G(I')\right]$ 
is an independent set in $G$ that contains more than $r+s$ vertices,
contradicting (\ref{eext0}).
Hence, the graph $F$ is a forest.

Let $T$ be a component of $F$ that contains a vertex, say $w_1$, from $W$.
By construction, the order of $T$ is at least $2$.
If $w_1$ is the only endvertex of $T$ from $W$,
then $I\Delta V(T)=(I\setminus V(T))\cup (V(T)\setminus I)$
is a maximum independent set of $G[X]$ containing less vertices from $D$,
contradicting the choice of $I$.
Hence, the tree $T$ contains an alternating induced path $P$
between $w_1$ and another vertex, say $w_t$, from $W$.
By symmetry, we may assume that
$P:w_1\bar{u}_1u_2\ldots u_t\bar{u}_tw_t$.

If the two bubbles $B_1$ and $B_t$ are distinct,
then $Q_1$, $P$, and $Q_t$ 
yield an induced alternating path with both endpoints in $\bar{D}_1$,
contradicting Claim \ref{claim1}.
Hence, we obtain that $B_1=B_t$.
Since the path $P_1$ and $P_t$ are disjoint,
the bubble $B_1$ must be a triangle $yy'z$,
the path $P_1$ ends in $y$,
the path $P_t$ ends in $y'$, and
$z$ belongs to $D$.
Since assuming that $B_1$ and $B_t$ differ leads to a contradiction,
the comment concerning the different choices for $P_i$ after Figure \ref{figpi}
implies that $P_1$ and $P_t$ both have length one,
that is, the paths $P_1$ and $P_t$ consist of the edges $w_1y$ and $w_ty'$, respectively.

If $t\geq 3$, then $B_2$ is distinct from $B_1$.
In this case, 
the paths $P_1$ or $P_t$, 
a suitable part of $P$, and $Q_2$
yield an induced alternating path with both endpoints in $\bar{D}_1$,
contradicting Claim \ref{claim1}.
Hence, we obtain that $t=2$,
that is, the length of $P$ is four.

Figure \ref{figs} illustrates the setup for the following arguments.

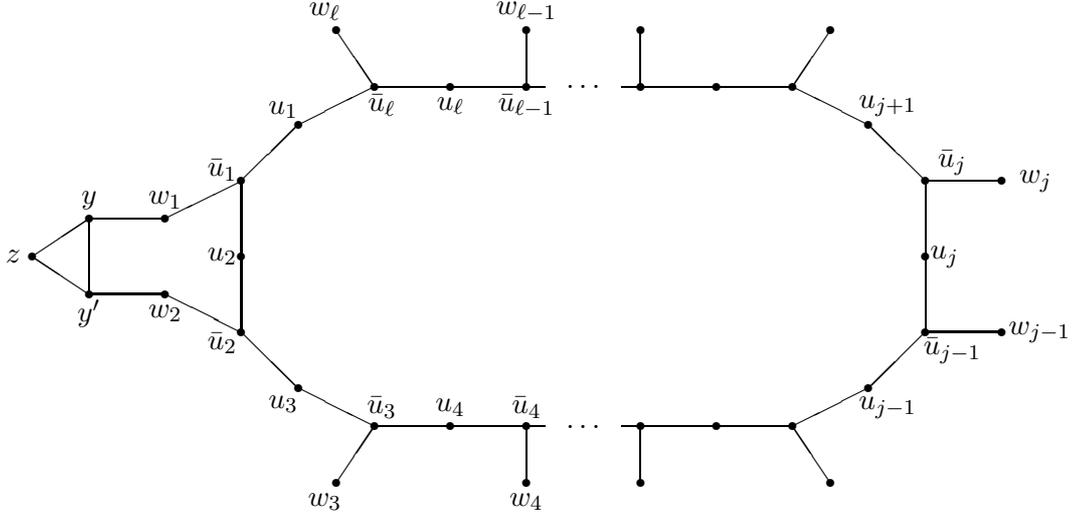
\begin{figure}[H]
\begin{center}
\unitlength 0.5mm 
\linethickness{0.4pt}
\ifx\plotpoint\undefined\newsavebox{\plotpoint}\fi 
\begin{picture}(260,130)(0,0)
\put(0,65){\circle*{2}}
\put(15,55){\circle*{2}}
\put(35,55){\circle*{2}}
\put(15,75){\circle*{2}}
\put(35,75){\circle*{2}}
\put(15,75){\line(0,-1){20}}
\put(15,55){\line(-3,2){15}}
\put(0,65){\line(3,2){15}}
\put(15,55){\line(1,0){20}}
\put(15,75){\line(1,0){20}}
\put(55,65){\circle*{2}}
\put(235,65){\circle*{2}}
\put(55,45){\circle*{2}}
\put(235,45){\circle*{2}}
\put(255,45){\circle*{2}}
\put(55,85){\circle*{2}}
\put(235,85){\circle*{2}}
\put(255,85){\circle*{2}}
\put(55,85){\line(-2,-1){20}}
\put(35,55){\line(2,-1){20}}
\put(55,85){\line(0,-1){20}}
\put(235,85){\line(0,-1){20}}
\put(55,65){\line(0,-1){20}}
\put(235,65){\line(0,-1){20}}
\put(70,100){\circle*{2}}
\put(220,100){\circle*{2}}
\put(70,30){\circle*{2}}
\put(220,30){\circle*{2}}
\put(55,85){\line(1,1){15}}
\put(235,85){\line(-1,1){15}}
\put(55,45){\line(1,-1){15}}
\put(235,45){\line(-1,-1){15}}
\put(90,110){\circle*{2}}
\put(200,110){\circle*{2}}
\put(90,20){\circle*{2}}
\put(200,20){\circle*{2}}
\put(70,100){\line(2,1){20}}
\put(220,100){\line(-2,1){20}}
\put(70,30){\line(2,-1){20}}
\put(220,30){\line(-2,-1){20}}
\put(90,110){\line(1,0){20}}
\put(200,110){\line(-1,0){20}}
\put(90,20){\line(1,0){20}}
\put(200,20){\line(-1,0){20}}
\put(110,110){\line(1,0){20}}
\put(180,110){\line(-1,0){20}}
\put(110,20){\line(1,0){20}}
\put(180,20){\line(-1,0){20}}
\put(110,110){\circle*{2}}
\put(180,110){\circle*{2}}
\put(110,20){\circle*{2}}
\put(180,20){\circle*{2}}
\put(130,110){\circle*{2}}
\put(160,110){\circle*{2}}
\put(130,20){\circle*{2}}
\put(160,20){\circle*{2}}
\put(130,110){\line(1,0){5}}
\put(160,110){\line(-1,0){5}}
\put(130,20){\line(1,0){5}}
\put(160,20){\line(-1,0){5}}
\put(145,110){\makebox(0,0)[cc]{$\ldots$}}
\put(145,20){\makebox(0,0)[cc]{$\ldots$}}
\put(80,125){\circle*{2}}
\put(80,5){\circle*{2}}
\put(210,125){\circle*{2}}
\put(210,5){\circle*{2}}
\put(130,125){\circle*{2}}
\put(130,5){\circle*{2}}
\put(160,125){\circle*{2}}
\put(160,5){\circle*{2}}
\put(130,125){\line(0,-1){15}}
\put(130,5){\line(0,1){15}}
\put(160,125){\line(0,-1){15}}
\put(160,5){\line(0,1){15}}
\put(80,125){\line(2,-3){10}}
\put(80,5){\line(2,3){10}}
\put(210,125){\line(-2,-3){10}}
\put(210,5){\line(-2,3){10}}
\put(234,85){\line(1,0){21}}
\put(235,45){\line(1,0){20}}
\put(15,80){\makebox(0,0)[cc]{$y$}}
\put(35,80){\makebox(0,0)[cc]{$w_1$}}
\put(35,50){\makebox(0,0)[cc]{$w_2$}}
\put(15,50){\makebox(0,0)[cc]{$y'$}}
\put(-5,65){\makebox(0,0)[cc]{$z$}}
\put(50,65){\makebox(0,0)[cc]{$u_2$}}
\put(240,65){\makebox(0,0)[cc]{$u_j$}}
\put(50,42){\makebox(0,0)[cc]{$\bar{u}_2$}}
\put(242,40){\makebox(0,0)[cc]{$\bar{u}_{j-1}$}}
\put(50,88){\makebox(0,0)[cc]{$\bar{u}_1$}}
\put(242,90){\makebox(0,0)[cc]{$\bar{u}_j$}}
\put(66,26){\makebox(0,0)[cc]{$u_3$}}
\put(225,25){\makebox(0,0)[cc]{$u_{j-1}$}}
\put(66,104){\makebox(0,0)[cc]{$u_1$}}
\put(225,105){\makebox(0,0)[cc]{$u_{j+1}$}}
\put(92,25){\makebox(0,0)[cc]{$\bar{u}_3$}}
\put(92,105){\makebox(0,0)[cc]{$\bar{u}_{\ell}$}}
\put(110,25){\makebox(0,0)[cc]{$u_4$}}
\put(110,105){\makebox(0,0)[cc]{$u_{\ell}$}}
\put(77,0){\makebox(0,0)[cc]{$w_3$}}
\put(77,130){\makebox(0,0)[cc]{$w_{\ell}$}}
\put(130,0){\makebox(0,0)[cc]{$w_4$}}
\put(130,130){\makebox(0,0)[cc]{$w_{\ell-1}$}}
\put(130,25){\makebox(0,0)[cc]{$\bar{u}_4$}}
\put(130,105){\makebox(0,0)[cc]{$\bar{u}_{\ell-1}$}}
\put(264,85){\makebox(0,0)[cc]{$w_j$}}
\put(265,45){\makebox(0,0)[cc]{$w_{j-1}$}}
\end{picture}
\end{center}
\caption{$B_1$, $P_1$, $P_2$, and $C$.}\label{figs}
\end{figure}
If the set $I'=\{ u_2,w_1,w_2,z\}$ is independent,
then the union of $I'$ 
with a maximal independent set in $G\left[\bar{D}\setminus N_G(I')\right]$ 
is an independent set in $G$ that contains more than $r+s$ vertices,
contradicting (\ref{eext0}).
Hence, by symmetry, and since $P$ is induced, 
we may assume that $z$ is adjacent to either $w_1$ or $u_2$.

If the neighbor of $w_1$ in $D$ does not lie in $X\setminus \{ w_2\}$,
then, using $P_1$, 
and the fact that $B_1$ is distinct from $B_{\ell},B_{\ell-1},\ldots,B_3$,
repeated applications of Claim \ref{claim1} imply in turn
that $w_{\ell}$ is adjacent to $u_1$,
that $w_{\ell-1}$ is adjacent to $u_{\ell}$,
that $w_{\ell-2}$ is adjacent to $u_{\ell-1}$, and so forth up to
that $w_3$ is adjacent to $u_4$.
Now, using the paths $y'w_2\bar{u}_2u_3\bar{u}_3w_3$ and $Q_3$,
Claim \ref{claim1} implies that $w_2$ is adjacent to $u_3$.
Yet now, the paths $yw_1\bar{u}_1u_2\bar{u}_2u_3\bar{u}_3w_3$
and $Q_3$ yield a contradiction to Claim \ref{claim1}.
Hence, by symmetry, 
the neighbors in $D$ of both vertices $w_1$ and $w_2$ lie in $X\setminus \{ w_1,w_2\}$.
This implies that the neighbor of $z$ in $D$ is $u_2$.

If $w_2$ is adjacent to $w_j$ for some $j\in \{ 3,\ldots,\ell\}$,
then, similarly as above, 
repeated applications of Claim \ref{claim1} imply that 
$u_3w_3,\ldots,u_{j-1}w_{j-1}$ as well as 
$u_1w_{\ell},\ldots,u_{j+2}w_{j+1}$ are edges in $G$,
and 
either $yw_1\bar{u}_1u_2\bar{u}_2u_3\ldots \bar{u}_j$
or $yw_1\bar{u}_1u_1\bar{u}_{\ell}u_{\ell}\ldots \bar{u}_j$
together with $Q_j$ yields a contradiction to Claim \ref{claim1}.
Hence, we may assume that $w_2$ is adjacent to $u_j$ 
for some $j\in [\ell]\setminus \{ 2\}$.
Repeated applications of Claim \ref{claim1} imply 
that $u_3w_3,\ldots,u_{j-1}w_{j-1}$ as well as 
$u_1w_{\ell},\ldots,u_{j+1}w_j$ are edges in $G$. 
Now, the paths $yw_1\bar{u}_1u_2\bar{u}_2u_3\ldots \bar{u}_j$
and $Q_j$ yield a contradiction to Claim \ref{claim1},
which completes the proof.
\end{proof}
Claim \ref{claim8} implies that ${\rm diss}(G[X])=2\alpha(G[X])$ for every $X$ in ${\cal X}$ with $|X|>6$. Hence, each $G[X]$ belongs to ${\cal G}$,
and, hence, arises from the disjoint union from copies of $K_3$ and $K_4^*$
by adding edges as specified for ${\cal G}$.

\begin{claim} \label{claim9}
For $X \in \mathcal{X}$ with $|X|=6$,
the subgraph $G[X]$ of $G$ 
is either $K_4^*$,
or arises from the disjoint union of two copies of $K_3$ by adding edges,
or is completely contained in an induced subgraph of $G$ that is isomorphic
to the graph shown in the left of Figure \ref{figb}.
\end{claim}
\begin{proof}[Proof of Claim \ref{claim9}]
Using the notation from the proof of Claim \ref{claim8}, 
let $X=\{ w_1,w_2,u_1, \bar{u}_1, u_2, \bar{u}_2\}$,
where $C:u_1\bar{u}_1u_2\bar{u}_2u_1$.
Similarly as in the proof of Claim \ref{claim8}, 
we obtain induced alternating paths $P_i$ and $Q_i$
between $w_i$ and some bubble $B_i$ for $i\in [2]$.

First, suppose that $u_1$ and $u_2$ are adjacent.
If $w_1$ and $w_2$ are adjacent, then $G[X]$ is a copy of $K_4^*$.
Hence, we may assume that $w_1$ and $w_2$ are not adjacent.
Claim \ref{claim1} implies that $B_1=B_2$, 
which implies that $B_1$ is a triangle $yy'z$,
$P_1$ ends in $y$, and $P_2$ ends in $y'$.
The observation after the definition of $P_i$ in Claim \ref{claim8} implies
that the paths $P_1$ and $P_2$ consist of the edges $w_1y$ and $w_2y'$,
respectively.
If the set $I'=\{ z,w_1,w_2,u_1\}$ is independent,
then the union of $I'$ 
with a maximal independent set in $G\left[\bar{D}\setminus N_G(I')\right]$ 
is an independent set in $G$ that contains more than $r+s$ vertices,
contradicting (\ref{eext0}).
Hence, by symmetry, we may assume that $z$ is adjacent to $w_1$,
and $G[X]$ is completely contained in the induced subgraph $G[X\cup V(B_1)]$ of $G$ 
that is isomorphic to the left graph shown in Figure \ref{figb}.
Hence, we may assume that $u_1$ and $u_2$ are not adjacent.

If $I'=\{ u_1,u_2\}\cup (V(Q_1)\cap D)$ is independent,
then the union of $I'$ 
with a maximal independent set in $G\left[\bar{D}\setminus N_G(I')\right]$ 
is an independent set in $G$ that contains more than $r+s$ vertices,
contradicting (\ref{eext0}).
Hence, by symmetry, we obtain 
that $w_1$ is adjacent to $u_1$,
and 
that $w_2$ is adjacent to $u_2$,
that is,
the graph $G[X]$ arises from the disjoint union of two copies of $K_3$ by adding edges.
This completes the proof of the claim.
\end{proof}
At this point of the proof
we know that $V(G)$ can be partitioned into sets $V_i$ such that 
each induced subgraph $G[V_i]$ is one of the graphs illustrated in Figure \ref{figpart},
where we also illustrate the possible intersections of $V_i$ with $D$ (up to symmetry).

\begin{figure}[H]
\begin{center}
\unitlength 0.6mm 
\linethickness{0.4pt}
\ifx\plotpoint\undefined\newsavebox{\plotpoint}\fi 
\begin{picture}(61,16)(0,0)
\put(0,0){\circle*{2}}
\put(40,0){\circle*{2}}
\put(20,0){\circle*{2}}
\put(60,0){\circle*{2}}
\put(10,15){\circle*{2}}
\put(50,15){\circle*{2}}
\put(10,15){\line(-2,-3){10}}
\put(50,15){\line(-2,-3){10}}
\put(0,0){\line(1,0){20}}
\put(40,0){\line(1,0){20}}
\put(20,0){\line(-2,3){10}}
\put(60,0){\line(-2,3){10}}
\put(0,0){\circle{4}}
\put(40,0){\circle{4}}
\put(20,0){\circle{4}}
\end{picture}\hspace{7mm}
\unitlength 0.6mm 
\linethickness{0.4pt}
\ifx\plotpoint\undefined\newsavebox{\plotpoint}\fi 
\begin{picture}(103,42)(0,0)
\put(1,0){\circle*{2}}
\put(41,0){\circle*{2}}
\put(81,0){\circle*{2}}
\put(1,20){\circle*{2}}
\put(41,20){\circle*{2}}
\put(81,20){\circle*{2}}
\put(1,40){\circle*{2}}
\put(41,40){\circle*{2}}
\put(81,40){\circle*{2}}
\put(21,0){\circle*{2}}
\put(61,0){\circle*{2}}
\put(101,0){\circle*{2}}
\put(21,20){\circle*{2}}
\put(61,20){\circle*{2}}
\put(101,20){\circle*{2}}
\put(21,40){\circle*{2}}
\put(61,40){\circle*{2}}
\put(101,40){\circle*{2}}
\put(1,0){\line(1,0){20}}
\put(41,0){\line(1,0){20}}
\put(81,0){\line(1,0){20}}
\put(1,40){\line(1,0){20}}
\put(41,40){\line(1,0){20}}
\put(81,40){\line(1,0){20}}
\put(1,40){\line(0,-1){40}}
\put(41,40){\line(0,-1){40}}
\put(81,40){\line(0,-1){40}}
\put(1,0){\line(1,1){20}}
\put(41,0){\line(1,1){20}}
\put(81,0){\line(1,1){20}}
\put(21,20){\line(0,-1){20}}
\put(61,20){\line(0,-1){20}}
\put(101,20){\line(0,-1){20}}
\put(21,0){\line(-1,1){20}}
\put(61,0){\line(-1,1){20}}
\put(101,0){\line(-1,1){20}}
\put(21,40){\line(0,-1){20}}
\put(61,40){\line(0,-1){20}}
\put(101,40){\line(0,-1){20}}
\put(1,0){\circle{4}}
\put(61,20){\circle{4}}
\put(21,0){\circle{4}}
\put(61,0){\circle{4}}
\put(101,0){\circle{4}}
\put(1,40){\circle{4}}
\put(41,40){\circle{4}}
\put(81,40){\circle{4}}
\put(101,40){\circle{4}}
\put(81,0){\circle{4}}
\end{picture}\\[7mm]
\unitlength .6mm 
\linethickness{0.4pt}
\ifx\plotpoint\undefined\newsavebox{\plotpoint}\fi 
\begin{picture}(262,39)(0,0)
\put(8,0){\circle*{2}}
\put(64,0){\circle*{2}}
\put(120,0){\circle*{2}}
\put(176,0){\circle*{2}}
\put(232,0){\circle*{2}}
\put(28,0){\circle*{2}}
\put(84,0){\circle*{2}}
\put(140,0){\circle*{2}}
\put(196,0){\circle*{2}}
\put(252,0){\circle*{2}}
\put(8,0){\line(1,0){20}}
\put(64,0){\line(1,0){20}}
\put(120,0){\line(1,0){20}}
\put(176,0){\line(1,0){20}}
\put(232,0){\line(1,0){20}}
\put(0,10){\circle*{2}}
\put(56,10){\circle*{2}}
\put(112,10){\circle*{2}}
\put(168,10){\circle*{2}}
\put(224,10){\circle*{2}}
\put(36,10){\circle*{2}}
\put(92,10){\circle*{2}}
\put(148,10){\circle*{2}}
\put(204,10){\circle*{2}}
\put(260,10){\circle*{2}}
\put(0,22){\circle*{2}}
\put(56,22){\circle*{2}}
\put(112,22){\circle*{2}}
\put(168,22){\circle*{2}}
\put(224,22){\circle*{2}}
\put(36,22){\circle*{2}}
\put(92,22){\circle*{2}}
\put(148,22){\circle*{2}}
\put(204,22){\circle*{2}}
\put(260,22){\circle*{2}}
\put(8,32){\circle*{2}}
\put(64,32){\circle*{2}}
\put(120,32){\circle*{2}}
\put(176,32){\circle*{2}}
\put(232,32){\circle*{2}}
\put(28,32){\circle*{2}}
\put(84,32){\circle*{2}}
\put(140,32){\circle*{2}}
\put(196,32){\circle*{2}}
\put(252,32){\circle*{2}}
\put(18,37){\circle*{2}}
\put(74,37){\circle*{2}}
\put(130,37){\circle*{2}}
\put(186,37){\circle*{2}}
\put(242,37){\circle*{2}}
\put(18,37){\line(-2,-1){10}}
\put(74,37){\line(-2,-1){10}}
\put(130,37){\line(-2,-1){10}}
\put(186,37){\line(-2,-1){10}}
\put(242,37){\line(-2,-1){10}}
\put(8,32){\line(-4,-5){8}}
\put(64,32){\line(-4,-5){8}}
\put(120,32){\line(-4,-5){8}}
\put(176,32){\line(-4,-5){8}}
\put(232,32){\line(-4,-5){8}}
\put(0,22){\line(0,-1){12}}
\put(56,22){\line(0,-1){12}}
\put(112,22){\line(0,-1){12}}
\put(168,22){\line(0,-1){12}}
\put(224,22){\line(0,-1){12}}
\put(0,10){\line(4,-5){8}}
\put(56,10){\line(4,-5){8}}
\put(112,10){\line(4,-5){8}}
\put(168,10){\line(4,-5){8}}
\put(224,10){\line(4,-5){8}}
\put(18,37){\line(2,-1){10}}
\put(74,37){\line(2,-1){10}}
\put(130,37){\line(2,-1){10}}
\put(186,37){\line(2,-1){10}}
\put(242,37){\line(2,-1){10}}
\put(28,32){\line(4,-5){8}}
\put(84,32){\line(4,-5){8}}
\put(140,32){\line(4,-5){8}}
\put(196,32){\line(4,-5){8}}
\put(252,32){\line(4,-5){8}}
\put(36,22){\line(0,-1){12}}
\put(92,22){\line(0,-1){12}}
\put(148,22){\line(0,-1){12}}
\put(204,22){\line(0,-1){12}}
\put(260,22){\line(0,-1){12}}
\put(36,10){\line(-4,-5){8}}
\put(92,10){\line(-4,-5){8}}
\put(148,10){\line(-4,-5){8}}
\put(204,10){\line(-4,-5){8}}
\put(260,10){\line(-4,-5){8}}
\multiput(0,22)(.055944056,-.153846154){143}{\line(0,-1){.153846154}}
\multiput(56,22)(.055944056,-.153846154){143}{\line(0,-1){.153846154}}
\multiput(112,22)(.055944056,-.153846154){143}{\line(0,-1){.153846154}}
\multiput(168,22)(.055944056,-.153846154){143}{\line(0,-1){.153846154}}
\multiput(224,22)(.055944056,-.153846154){143}{\line(0,-1){.153846154}}
\multiput(36,22)(-.055944056,-.153846154){143}{\line(0,-1){.153846154}}
\multiput(92,22)(-.055944056,-.153846154){143}{\line(0,-1){.153846154}}
\multiput(148,22)(-.055944056,-.153846154){143}{\line(0,-1){.153846154}}
\multiput(204,22)(-.055944056,-.153846154){143}{\line(0,-1){.153846154}}
\multiput(260,22)(-.055944056,-.153846154){143}{\line(0,-1){.153846154}}
\multiput(36,10)(-.055944056,.153846154){143}{\line(0,1){.153846154}}
\multiput(92,10)(-.055944056,.153846154){143}{\line(0,1){.153846154}}
\multiput(8,32)(-.055944056,-.153846154){143}{\line(0,-1){.153846154}}
\multiput(64,32)(-.055944056,-.153846154){143}{\line(0,-1){.153846154}}
\multiput(120,32)(.0714285714,-.056122449){392}{\line(1,0){.0714285714}}
\multiput(176,32)(.0714285714,-.056122449){392}{\line(1,0){.0714285714}}
\multiput(232,32)(.0714285714,-.056122449){392}{\line(1,0){.0714285714}}
\multiput(140,32)(-.0714285714,-.056122449){392}{\line(-1,0){.0714285714}}
\multiput(196,32)(-.0714285714,-.056122449){392}{\line(-1,0){.0714285714}}
\multiput(252,32)(-.0714285714,-.056122449){392}{\line(-1,0){.0714285714}}
\put(0,10){\circle{4}}
\put(0,22){\circle{4}}
\put(36,22){\circle{4}}
\put(36,10){\circle{4}}
\put(18,37){\circle{4}}
\put(74,37){\circle{4}}
\put(56,10){\circle{4}}
\put(64,0){\circle{4}}
\put(92,22){\circle{4}}
\put(92,10){\circle{4}}
\put(112,10){\circle{4}}
\put(112,22){\circle{4}}
\put(148,22){\circle{4}}
\put(148,10){\circle{4}}
\put(130,37){\circle{4}}
\put(186,37){\circle{4}}
\put(242,37){\circle{4}}
\put(204,22){\circle{4}}
\put(204,10){\circle{4}}
\put(260,22){\circle{4}}
\put(260,10){\circle{4}}
\put(176,0){\circle{4}}
\put(168,10){\circle{4}}
\put(232,0){\circle{4}}
\put(224,22){\circle{4}}
\end{picture}
\end{center}
\caption{The induced subgraphs $G[V_i]$ and their possible intersection with $D$;
the vertices that belong to $D$ are encircled.}\label{figpart}
\end{figure}
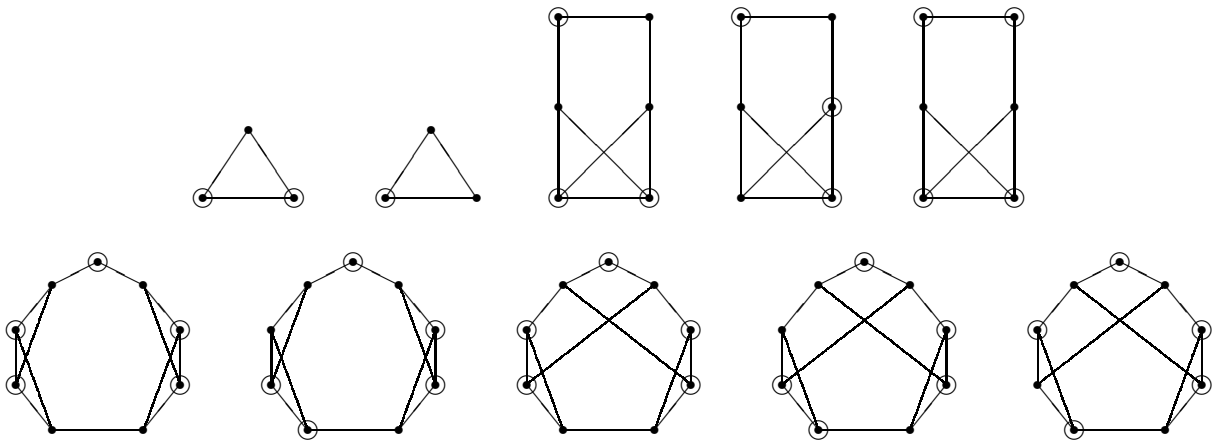
Let $n_3$ be the number of indices $i$ such that $G[V_i]$ is a triangle,
and let $k_3$ be the number of those $G[V_i]$ with $|V_i\cap D|=1$.
Let $n_2$ be the number of indices $i$ such that $G[V_i]$ is a $K_4^*$,
and let $k_2$ be the number of those $G[V_i]$ with $|V_i\cap D|=3$.
Let $n_1$ be the number of indices $i$ such that $G[V_i]$ has order $9$.
Contracting each $G[V_i]$ to a single vertex yields  
a connected subcubic multigraph $H$.
Since 
$18k=n=3n_3+6n_2+9n_1$,
we have 
$6k=n_3+2n_2+3n_1$.
The edges of $H$ that correspond to edges of $G$ between vertices from $D$
that belong to different sets $V_i$ form a matching $M$ in $H$
with $2|M|=k_3+k_2+n_1$ that covers all vertices of degree $1$ in $H$.
Since $10k=|D|=2n_3-k_3+4n_2-k_2+6n_1-n_1$, 
we obtain $k_2+k_3+n_1=2(n_3+2n_2+3n_1)-10k=2k$,
that is, the matching $M$ of $H$ has size exactly $k$.
The multigraph $H$ has exactly $\frac{1}{2}(3n_3+2n_2+n_1)$ edges.
Furthermore, the multigraph $H$ has exactly $2(n_3-k_3)+2(n_2-k_2)$ edges $e$
that correspond to an edge $uv$ of $G$ such that $u\in D$ and $v\in \bar{D}$,
and $u$ and $v$ lie in different $V_i$.
Since 
$2(n_3-k_3)+2(n_2-k_2)+|M|=\frac{1}{2}(3n_3+2n_2+n_1)$,
the multigraph $H$ actually has no edge $e$
that corresponds to an edge $uv$ of $G$ such that $u,v\in \bar{D}$,
and $u$ and $v$ lie in different $V_i$.
This implies that the matching $M$ is induced.
Furthermore, it implies that orienting the edges $e$ of $H$ 
that do not belong to $M$ away from the set $V_i$ 
containing the element of $D$ in $e$
yields an orientation $\overrightarrow{H-M}$ of $H-M$ 
such that every vertex of $H$ that is not incident with an edge in $M$ 
has exactly two outgoing edges.
Altogether, it follows that $G$ belongs to ${\cal G}_k$,
which completes the proof.
\end{proof}
While Theorem \ref{conjecture1} provides deep structural insights
concerning the extremal graphs for Theorem \ref{theorem1},
it remains unclear whether it leads to a polynomial time recognition algorithm 
for these graphs.

\end{document}